\documentclass[reqno]{amsart}
\usepackage{amsthm,amsmath,amsfonts,amssymb,enumerate,latexsym,mathrsfs,lineno}

\usepackage{graphicx}
\newtheorem{thm}{Theorem}

\newtheorem*{thm*}{Theorem}

\newtheorem{defn}{Definition}
\newtheorem*{defn*}{Definition}

\newtheorem{claimm}{Claim}

\newtheorem{lemma}{Lemma}

\newtheorem*{cor*}{Corollary}

\newcommand{\drawat}[3]{\makebox[0pt][l]{\raisebox{#2}{\hspace*{#1}#3}}}

\newcommand{\N}{\mathbb{N}}
\newcommand{\R}{\mathbb{R}}
\newcommand{\Q}{\mathbb{Q}}

\newcommand{\iso}{\mathrm{iso}}
\newcommand{\funct}[2]{#1 \longrightarrow #2}

\newcommand{\restrict}[2]{#1\mspace{-2mu}\mathbin{\upharpoonright}\mspace{-1mu} #2}

\newcommand{\E}{\mathcal{E}}

\newcommand{\Ur}{\textbf{U}}

\newcommand{\s}{\mathbb{S}^{\infty}}

\newcommand{\dom}{\mbox{$\mathrm{dom}$}}

\newcommand{\m}[1]{\textbf{#1}}

\newcommand{\mc}[1]{\widetilde{\textbf{#1}}}

\author{L. Nguyen Van Th\'{e} and N. W. Sauer}

\address{Department of Mathematics and Statistics, University of Calgary, 2500 University Drive NW, Calgary, Alberta, Canada, T2N1N4}
\address{and}
\address{Institut de math\'ematiques, Universit\'e de Neuch\^atel, Emile-Argand 11, 2000 Neuch\^atel, Switzerland}

\email{nguyen@math.ucalgary.ca}

\address{Department of Mathematics and Statistics, University of Calgary, 2500 University Drive NW, Calgary, Alberta, Canada, T2N1N4.}

\email{nsauer@math.ucalgary.ca}

\title{Some weak indivisibility results in ultrahomogeneous metric spaces.}

\subjclass[2000]{Primary: 03E02. Secondary: 05C55, 05D10, 51F99}
\keywords{Ramsey theory, Metric geometry, Urysohn metric space}
\date{January, 2010}

\begin{document}

\begin{abstract}
We study the validity of a partition property known as weak indivisibility for the integer and the rational Urysohn metric spaces. We also compare weak indivisiblity to another partition property, called age-indivisibility, and provide 
an example of a countable ultrahomogeneous metric space which may be age-indivisible but not weakly indivisible.

\end{abstract}

\maketitle

\section{Introduction.}

The purpose of this article is the study of certain partition properties of particular metric spaces, called \emph{ultrahomogeneous}. A metric space $\m{X}$ is ultrahomogeneous when every isometry between finite
metric subspaces of $\m{X}$ can be extended to an isometry of $\m{X}$ onto itself. For example, when seen as a metric space, any Euclidean space $\R^n$ has this property. So do the separable infinite dimensional Hilbert space $\ell_2$ and its unit sphere $\s$. Another less known example of ultrahomogeneous metric space, though recently a well studied object (see \cite{Pr}), is the \emph{Urysohn space}, denoted $\Ur$: up to isometry, it is the unique complete separable ultrahomogeneous metric space into which every separable metric space embeds (here and in the rest of the paper, all the embeddings are \emph{isometric}, that is, distance preserving). This space also admits numerous countable analogs. For example, for various countable sets $S$ of positive reals (see \cite{DLPS} for the precise condition on $S$), there is, up to isometry, a unique countable ultrahomogeneous metric space into which every countable metric space with distances in $S$ embeds. When $S=\Q$ or $\N$ this gives raise to the  spaces denoted respectively $\Ur_{\Q}$ (the \emph{rational Urysohn space}) and $\Ur_{\N}$ (the \emph{integer Urysohn space}). Recently, separable ultrahomogeneous metric spaces have been at the center of active research because of a remarkable connection between their combinatorial behavior when submitted to finite partitions and the dynamical properties of their isometry group. For example, consider the following result. Call a metric space $\m{Z}=(Z,d^{\m{Z}})$ \emph{age-indivisible} if for every finite metric subspace $\m{Y}$ of $\m{Z}$ and every partition $Z = B\cup R$ of the underlying set $Z$ of $\m{Z}$ (thought as a coloring of the points of $Z$ with two colors, blue and red), the space $\m{Y}$ embeds in $B$ or $R$ (here and in the whole paper, boldface characters will refer to metric structures while lightface characters will refer to the corresponding underlying sets).

\begin{thm*}[Folklore]
The spaces $\Ur_{\Q}$ and $\Ur_{\N}$ are age-indivisible.
\end{thm*} 

There are at least two directions for possible generalizations. First, one may ask what happens if instead of coloring the points of, say, the space $\Ur_{\Q}$, we color the isometric copies of a fixed finite metric subspace $\m{X}$ of $\Ur_{\Q}$. We will not touch this subject here but Kechris, Pestov and Todorcevic showed in \cite{KPT} that the answer to this question (obtained by Ne\v{s}et\v{r}il in \cite{N1}) has spectacular consequences on the groups $\iso(\Ur_{\Q})$ and $\iso(\Ur)$ of surjective self-isometries of $\Ur_{\Q}$ and $\Ur$. For example, every continuous action of $\iso(\Ur)$ (equipped with the pointwise convergence topology) on a compact topological space admits a fixed point. 

Another direction of generalization is to ask whether any of those spaces is \emph{indivisible}, that is, whether $B$ or $R$ necessarily contains not only a copy of a fixed finite $\m{Y}$ but of the whole space itself. However, it is known that any indivisible metric space must have a bounded distance set. Therefore, the spaces $\Ur_{\Q}$ and $\Ur_{\N}$ are not indivisible. Still, in this article, we investigate whether despite this obstacle, a partition result weaker than indivisibility but stronger than age-indivisibility holds. Call a metric space $\m{X}$ \emph{weakly indivisible} when for every finite metric subspace $\m{Y}$ of $\m{X}$ and every partition $X = B \cup R$, either $\m{Y}$ embeds in $B$ or $\m{X}$ embeds in $R$. Building on techniques developed in \cite{LANVT} and \cite{NVTS}, we prove: 





\begin{thm}
\label{thm:Uomega}
The space $\Ur_{\N}$ is weakly indivisible. 
\end{thm}

As for $\Ur_{\Q}$, we are not able to prove or disprove weak indivisibility but we obtain the following weakening as a consequence of Theorem \ref{thm:Uomega}. If $\m{X}$ is a metric space, $Y \subset X$ and $\varepsilon > 0$, $(Y)_{\varepsilon}$ denotes the set
\[(Y) _{\varepsilon} = \{ x \in X : \exists y \in Y \ \ d ^{\m{X}} (x,y) \leq \varepsilon \}.\]

\begin{thm}
\label{thm:UQ}
Let $U_{\Q} = B \cup R$ and $\varepsilon > 0$. Assume that there is a finite metric subspace $\m{Y}$ of $\Ur_{\Q}$ that does not embed in $B$. Then $\Ur_{\Q}$ embeds in $(R)_{\varepsilon}$. 
\end{thm}

Another consequence of Theorem \ref{thm:Uomega} is the following partition result for $\Ur$.

\begin{thm}
\label{thm:U}
Let $U = B \cup R$ and $\varepsilon > 0$. Assume that there is a compact metric subspace $\m{K}$ of $\Ur$ that does not embed in $(B)_{\varepsilon}$. Then $\Ur$ embeds in $(R)_{\varepsilon}$.  
\end{thm}

Note that these results do not answer the following: for a countable ultrahomogeneous metric space is weak indivisibility a strictly stronger property than age-indivisibility? In the last section of this paper, we indicate an example of a countable ultrahomogeneous metric space which might be age-indivisible but not weakly indivisible. To our knowledge, this could be one of the first two known examples of a countable ultrahomogeneous relational structure witnessing that weak indivisibility and age-indivisiblity are distinct properties (the other example will appear in \cite{LNS}). Let $\E _{\Q}$ be the class of all finite metric spaces $\m{X}$ with distances in $\Q$ which embed isometrically into the unit sphere $\mathbb{S} ^{\infty}$ of $\ell _2$ with the property that $\{ 0_{\ell _2} \} \cup \m{X}$ is affinely independent. It is known that there is a unique countable ultrahomogeneous metric space $\s _{\Q}$ whose class of finite metric subspaces is exactly $\E _{\Q}$, and that the metric completion of $\s _{\Q}$ is $\s$ (for a proof, see \cite{NVT} or \cite{NVT'}). 

\begin{thm}
\label{thm:ageind'}
The space $\s _{\Q}$ is age-indivisible. 
\end{thm}

The proof of this result requires the use of a deep theorem due to Matou\v{s}ek and R\"odl in Euclidean Ramsey theory. As for the negation of weak indivisibility of $\s _{\Q}$, its proof is conditioned by the validity of a strong form of the Odell-Schlumprecht distortion theorem in Banach space theory, see Section \ref{section:age not weak} for more details. 

The paper is organized as follows. In Section \ref{section:Uomega}, we prove Theorem \ref{thm:Uomega}. In Section \ref{section:U}, we prove Theorem \ref{thm:U}. Theorem \ref{thm:UQ} is proved in section \ref{section:UQ}, and 
Theorem \ref{thm:ageind'} is proved in Section \ref{section:age not weak}, where a discussion of weak indivisibility of $\s _{\Q}$ is also included. 

\

\textbf{Acknowledgements.} We would like to thank Thomas Schlumprecht for helpful discussions concerning the potential negation of weak indivisibility of $\s _{\Q}$. We would also like to express our gratitude to three anonymous referees, thanks to whom the quality of this paper was considerably improved. L. Nguyen Van Th\'e would like to acknowledge the support of the Department of Mathematics \& Statistics Postdoctoral Program at the University of Calgary, as well as the support of the Swiss National Fund via the postdoctoral fellowship \# 20-118014-1. N. W. Sauer was supported by NSERC of Canada Grant \# 691325.

\section{Proof of Theorem \ref{thm:Uomega}} \label{section:Uomega}

The purpose of this section is to prove Theorem \ref{thm:Uomega}. In fact, we prove a slightly stronger result. We mentioned in introduction that there are various countable sets $S$ of positive reals for which there is, up to isometry, a unique countable ultrahomogeneous metric space into which every countable metric space with distances in $S$ embeds. It can be proved that when $p\in \N$, the integer interval $\{1,\ldots,p \}$ is such a set. The corresponding countable ultrahomogeneous metric space is denoted $\Ur_p$. 

\begin{thm}
\label{thm:Uomega'}
Let $U_{\N} = B \cup R$. Assume that there is $p \in \N$ such that $\Ur_p$ does not embed in $B$. Then $\Ur_{\N}$ embeds in $R$. 
\end{thm}

This section is devoted to the proof of Theorem \ref{thm:Uomega'}. We fix $p \in \N$ as well as a partition $U_{\N} = B \cup R$ such that $\Ur_p$ does not embed in $B$. Our goal is to prove that $\Ur_{\N}$ embeds into $R$. Let $m:=\left\lceil p/2\right\rceil$ (the least integer larger than or equal to $p/2$). Recall that if $Y\subset U_{\N}$, the set $(Y)_{\varepsilon}$ is defined by \[(Y) _{\varepsilon} = \{ x \in U_{\N} : \exists y \in Y \ \ d ^{\Ur_{\N}} (x,y) \leq \varepsilon \}.\] 

In particular, if $x\in U_{\N}$, the set $(\{ x \})_{m-1}$ denotes the set of all elements of $U_{\N}$ at distance $\leq m-1$ from $x$. We are going to construct $\tilde{U} \subset R$ isometric to $\Ur_{\N}$ recursively such that for every $x \in \tilde{U}$, \[ (\{ x \})_{m-1} \cap \tilde{U} \subset R.\] 

More precisely, fix an enumeration $\{ x_n : n \in \N\}$ of $U_{\N}$. We are going to construct $\{\tilde{x}_n: n \in \N\} \subset U_{\N}$ recursively together with a decreasing sequence $\m{D}_0 , \m{D}_1 , \ldots$ of metric subspaces of $\Ur_{\N}$ such that $x_n \mapsto \tilde{x}_n$ is an isometry and, for every $n \in \N$, each $\m{D}_n$ is isometric to $\Ur_{\N}$, $\{\tilde{x}_k: k\leq n \} \subset D_n$, and $(\{ \tilde{x}_n \})_{m-1} \cap D_n \subset R$. To do so, we will need the notion of \emph{Kat\v{e}tov} map as well as several technical lemmas. 

\begin{defn}
Given a metric space $\m{X} = (X, d^{\m{X}})$, a map $f:\funct{X}{(0,+\infty)}$ is \emph{Kat\v{e}tov over $\m{X}$} when \[\forall x, y \in X, \ \ |f(x) - f(y)| \leq d^{\m{X}} (x,y) \leq f(x) + f(y).\] 
\end{defn}

Equivalently, one can extend the metric $d^{\m{X}}$ to $X \cup \{ f \}$ by defining, for every $x, y$ in $X$, $ \widehat{d^{\m{X}}} (x, f) = f(x)$ and $\widehat{d^{\m{X}}} (x, y) = d^{\m{X}} (x, y)$. The corresponding metric space is then written $\m{X} \cup \{ f
\}$. The set of all Kat\v{e}tov maps over $\m{X}$ is written $E(\m{X})$. For a metric subspace $\m{X}$ of $\m{Y}$ and a Kat\v{e}tov map $f \in E(\m{X})$, the \emph{orbit of $f$ in $\m{Y}$} is the set $O(f,\m{Y})$ defined by \[ O(f,\m{Y}) = \{ y \in Y : \forall x \in X \ \ d^{\m{Y}}(y,x) = f(x)\}. \]

Any element $y \in O(f, \m{Y})$ is said to \emph{realize} $f$ over $\m{X}$. Here, the concepts of Kat\v{e}tov map and orbit are relevant because of the following standard reformulation of the notion of ultrahomogeneity, which will be used in the rest of the paper: 

\begin{lemma}
\label{prop:extension} Let $\m{X}$ be a countable metric space. Then $\m{X}$ is ultrahomogeneous
iff for every finite subspace  $\m{F}$ of $\m{X}$ and every Kat\v{e}tov map $f$ over $\m{F}$, if
$\m{F} \cup \{ f \}$ embeds into $\m{X}$, then $O(f, \m{X}) \neq \emptyset$.
\end{lemma}

\begin{proof} Postponed to section \ref{subsection:extension}. \end{proof}

\begin{lemma}

\label{lem:mono}

Let $G$ be a finite subset of $U_{\N}$, and $g$ a Kat\v{e}tov map with domain $G$ and with values in $\N$. Then there exists an isometric copy $\m{C}$ of $\Ur _{\N}$ inside $\Ur _{\N}$ such that:
\begin{enumerate} 
\item $G \subset C$,
\item $O(g, \m{C})\subset B$ or $O(g, \m{C})\subset \ R$.  
\end{enumerate}
\end{lemma}

In words, Lemma \ref{lem:mono} asserts that going to a subcopy of $\Ur_{\N}$ if necessary, we may assume that the orbit of $g$ is completely included in one of the parts of the partition. Observe that as a metric space, the orbit $O(g, \m{C})$ is isometric to $\Ur _n$ where $n = 2\min g$ (indeed, it is countable ultrahomogeneous with distances in $\{1,\ldots ,n \}$ and embeds every countable metric space with distances in $\{1,\ldots ,n \}$). 

\begin{proof}
The proof of Lemma \ref{lem:mono} can be found in \cite{NVTS}. More precisely, Lemma \ref{lem:mono} can be obtained by combining Lemma 2 \cite{NVTS} and Lemma 3 \cite{NVTS} after having replaced $\Ur_p$ by $\Ur_{\N}$ in these statements. The proof of Lemma 3 \cite{NVTS} is elementary, while the proof of Lemma 2 \cite{NVTS} represents the core of \cite{NVTS} and is too lengthy to be presented here. These two proofs can be carried out with no modification once $\Ur_p$ has been replaced by $\Ur_{\N}$.  
\end{proof}

\begin{lemma}

\label{lem:red}

Let $G_0 \subset G$ be finite subsets of $U_{\N}$, and let $\mathcal{G}$ a finite family of Kat\v{e}tov maps with domain $G$ and such that for all $g, g' \in \mathcal{G}$: \[ \max(\restrict{|g - g'|}{G_0}) = \max | g - g'|, \] \[ \min(\restrict{(g+g')}{G_0}) = \min(g + g'),\] \[ \min(\restrict{g}{G_0})=\min(g). \]

Then there exists an isometric copy $\m{C}$ of $\Ur _{\N}$ inside $\Ur _{\N}$ such that:
\begin{enumerate} 
\item $G \cap C = G_0$,
\item $\forall g \in \mathcal{G} \ \ O(\restrict{g}{G_0} , \m{C}) \subset O(g, \Ur _{\N}).$  
\end{enumerate} 

\end{lemma}

\begin{proof} Postponed to section \ref{subsection:red}. \end{proof}

\subsection{Construction of $\tilde{x}_0$ and $\m{D}_0$}



First, pick an arbitrary $u \in U_{\N}$ and consider the map $g : \funct{\{u\}}{\N}$ defined by $g(u)=m$. By Lemma \ref{lem:mono}, find an isometric copy $\m{C}$ of $\Ur _{\N}$ inside $\Ur _{\N}$ such that:
\begin{enumerate} 
\item $u \in C$,
\item $O(g, \m{C})\subset B$ or $O(g, \m{C})\subset \ R$.  
\end{enumerate}

Note that since $g$ has minimum $m$, the orbit $O(g, \m{C})$ is isometric to $\Ur_{2m}$ and therefore contains a copy of $\Ur_p$. Hence, because $\Ur_p$ does not embed in $B$, the inclusion $O(g, \m{C})\subset B$ is excluded and we really have $O(g, \m{C})\subset R$. Let $\tilde{x}_0~\in~O(g, \m{C})$ and for every $k\leq m$ let $g_k:\funct{\{ u, \tilde{x}_0 \}}{\N}$ be such that $g_k(u)=m$ and $g_k(\tilde{x}_0)=k$. The sets $G_0 = \{ \tilde{x}_0 \}$ and $G = \{ u, \tilde{x}_0 \}$, and the family $\mathcal{G}=\{ g_k:k\leq m\}$ satisfy the hypotheses of Lemma \ref{lem:red}, which allows us to obtain an isometric copy $\m{D}_0$ of $\Ur _{\N}$ inside $\m{C}$ such that:
\begin{enumerate} 
\item $\{ u, \tilde{x}_0 \} \cap D_0 = \{ \tilde{x}_0 \}$,
\item $\forall k\leq m \ \ O(\restrict{g_k}{\{ \tilde{x}_0 \}} , \m{D}_0) \subset O(g_k, \m{C})$.   
\end{enumerate}

Note that for every $k\leq m$, we have $O(g_k, \m{C}) \subset O(g, \m{C})\subset R$. Therefore, in $\m{D}_0$, all the spheres around $\tilde{x}_0$ with radius $k\leq m$ are included in $R$. So \[(\{ \tilde{x}_0 \})_{m-1} \cap D_0 \subset R. \ \ \qed \]

\subsection{Induction step}

Assume that we constructed $\{\tilde{x}_k: k \leq n \} \subset U_{\N}$ together with a decreasing sequence $\m{D}_0 ,\ldots, \m{D}_n$ of metric subspaces of $\Ur_{\N}$ such that $x_k \mapsto \tilde{x}_k$ is an isometry (recall that $\{ x_n : n \in \N\}$ is the enumeration of $U_{\N}$ we fixed at the beginning of the proof), each $\m{D}_k$ is isometric to $\Ur_{\N}$, $\{\tilde{x}_k: k\leq n \} \subset D_n$ and $(\{ \tilde{x}_k \})_{m-1} \cap D_n \subset R$ for every $k \leq n$. We are going to construct $\tilde{x}_{n+1}$ and $\m{D}_{n+1}$. Consider the map $f : \funct{\{ \tilde{x}_0,\ldots,\tilde{x}_n \}}{\N}$ where \[ \forall k\leq n \ \ f(\tilde{x}_k) = d^{\Ur_{\N}}(x_k, x_{n+1}).\] 

Recalling that $E(\{ \tilde{x}_0,\ldots,\tilde{x}_n \})$ denotes the set of all Kat\v{e}tov maps from the set $\{ \tilde{x}_0,\ldots,\tilde{x}_n \}$ to $\N$, consider the set $\mathcal{G}$ defined by \[ \{ g \in E(\{ \tilde{x}_0,\ldots,\tilde{x}_n \}) : \forall k \leq n \ (\left|f(\tilde{x}_k) - g(\tilde{x}_k))\right| \leq m-1 \ \ \textrm{and} \ \ g(\tilde{x}_k) \geq m)\}. \] 

This set is finite and a repeated application of Lemma \ref{lem:mono} allows us to construct an isometric copy $\m{D}_{n+1}$ of $\Ur _{\N}$ inside $\m{D}_n$ such that:
\begin{enumerate} 
\item $\{ \tilde{x}_0,\ldots,\tilde{x}_n \} \subset D_{n+1}$,
\item $\forall g \in \mathcal{G}, \ \ O(g, \m{D}_{n+1})\subset B \ \textrm{or} \ R$.  
\end{enumerate}

Note that since every $g \in \mathcal{G}$ has minimum $m$, the orbit $O(g, \m{D}_{n+1})$ is isometric to $\Ur_{2m}$ and therefore contains a copy of $\Ur_p$. Because $\Ur_p$ does not embed in $B$, we consequently have \[ \forall g \in \mathcal{G}, \ \ O(g, \m{D}_{n+1})\subset R. \]

Let $\tilde{x}_{n+1} \in O(f, \m{D}_{n+1})$. We claim that $\tilde{x}_{n+1}$ and $\m{D}_{n+1}$ are as required. Note that, because $\tilde{x}_{n+1} \in O(f, \m{D}_{n+1})$, we have \[ \forall k\leq n \ \ d^{\Ur_{\N}}(\tilde{x}_{n+1}, \tilde{x}_k) = f(\tilde{x}_k) = d^{\Ur_{\N}}(x_k, x_{n+1}).\] 

Therefore, $x_k \mapsto \tilde{x}_k$ is an isometry. Next we prove that $(\{ \tilde{x}_{n+1} \})_{m-1} \cap D_{n+1} \subset R$. Indeed, let $y\in (\{ \tilde{x}_{n+1} \})_{m-1} \cap D_{n+1}$. If $d^{\Ur_{\N}}(\tilde{x}_k,y) \geq  m$ for every $k\leq n$, then the map $d^{\Ur_{\N}}(\cdot,y)$ is in $\mathcal{G}$ and so $y \in O(d^{\Ur_{\N}}(\cdot,y), \m{D}_{n+1})\subset R$. Otherwise, we have $d^{\Ur_{\N}}(\tilde{x}_k,y) < m$ for some $k\leq n$ and \[ y \in \ (\{ \tilde{x}_k \})_{m-1} \cap D_{n+1} \ \subset \ (\{ \tilde{x}_k \})_{m-1} \cap D_{n} \ \subset R. \ \ \qed \]

\subsection{Proof of Lemma \ref{prop:extension}}

\label{subsection:extension}

The proof is standard but we detail it here for completeness. Assume that $\m{X}$ is ultrahomogeneous. Let $\varphi : \funct{\m{F} \cup \{ f \}}{\m{X}}$ be an embedding. By ultrahomogeneity of $\m{X}$, there is an isometry $\psi$ of $\m{X}$ onto itself such that: \[ \forall x \in F, \ \ \psi (x) = \varphi (x). \]

Then, the point $\psi ^{-1} (\varphi (f))$ is in $O(f, \m{X})$.  

For the converse, assume that $\{ x_0 ,\ldots, x_n \}$ and $\{z_0 ,\ldots, z_n \}$ are isometric finite subspaces of $\m{X}$ and that $\varphi : x_k \mapsto z_k$ is an isometry. We wish to extend $\varphi$ to an isometry of $\m{X}$ onto itself. We do that thanks to a back and forth method. First, extend $\{ x_0 ,\ldots, x_n \}$ and $\{z_0 ,\ldots, z_n \}$ so that $\{ x_k : k \in \N \} = \{ z_k : k \in \N \} = X$. For $k\leq n$, let $\sigma (k) = \tau (k) = k$. Then, set $\sigma (n+1) = n+1$. Consider the map $f _{n+1}$ defined on $\{\varphi ( x_{\sigma (k)}) : k \leq n \}$ by:

\begin{center}
$\forall k \leq n, \ \ f_{n+1} (\varphi (x_{\sigma (k)}))  = d^{\m{X}}(x_{\sigma (n+1)} , x _{\sigma (k)})$.
\end{center}

Observe that $f_{n+1}$ is Kat\v{e}tov over $\{\varphi ( x_{\sigma (k)}) : k \leq n \}$ and that the space $\{\varphi ( x_{\sigma (k)}) : k \leq n \} \cup \{ f_{n+1}\}$ is isometric to $\{ x_{\sigma (k)} : k \leq n+1 \}$. By hypothesis on $\m{X}$, we can consequently find a point in $O(f_{n+1}, \m{X})$, call it $\varphi (x_{\sigma (n+1)})$. Next, set: \[ \tau (n+1) = \min \{ k \in \N : z_k \notin \{ \varphi (x_{\sigma (i)}) : i \leq n \} \}.\]

Consider the map $g_{n+1}$ defined on $\{ x_{\sigma (k)} : k \leq n \}$ by: \[\forall k \leq n, \ \ g_{n+1} (x_{\sigma (k)}) = d^{\m{X}}(z_{\tau (n+1)} , \varphi (x_{\sigma (k)})).\]

Then $g_{n+1}$ is Kat\v{e}tov over the space $\{x_{\sigma (k)} : k \leq n \}$ and the corresponding union $\{x_{\sigma (k)} : k \leq n \} \cup \{ g_{n+1}\}$ is isometric to $\{ \varphi (x_{\sigma (k)}) : k \leq n \} \cup \{ z_{\tau (n+1)} \}$. So again, by hypothesis on $\m{X}$, we can find a point in $O(g_{n+1},\m{X})$, call it $\varphi ^{-1} (z_{\tau (n+1)})$. In general, if $\sigma$ and $\tau$ have been defined up to $m$ and $\varphi$ has been defined on $T_m := \{x_{\sigma (0)},\ldots , x_{\sigma (m)} \} \cup \{ \varphi ^{-1} (z_{\sigma (0)}),\ldots , \varphi (z_{\sigma (m)}) \}$, set: \[\sigma (m+1) = \min \{ k \in \N : x_k \notin T_m \}.\]

Consider the map $f _{m+1}$ defined on $\varphi (T_m)$ by:

\[
\forall k \leq m, \ \left\{ 
\begin{array}{l} 
f_{m+1} ( \varphi (x _{\sigma (k)})) = d^{\m{X}}(x_{\sigma (m+1)} , x _{\sigma (k)})\\ 
f_{m+1} ( z _{\tau (k)})) = d^{\m{X}}(x_{\sigma (m+1)} ,\varphi ^{-1} (z _{\tau (k)}) )
\end{array} \right.
\]

Observe that $f_{m+1}$ is Kat\v{e}tov over $\varphi (T_m)$ and that $\varphi (T_m \cup \{ f_{m+1}\})$ is isometric to $T_m \cup \{ x_{\sigma (m+1)} \}$. By hypothesis on $\m{X}$, we can consequently find a point in $O(f_{m+1},\m{X})$, call it $\varphi (x_{\sigma (m+1)})$. Next, let:

\begin{center}
$\tau (m+1) = \min \{ k \in \N : z_k \notin \{ \varphi (x_{\sigma (i)}) : i < n+1 \} \}$ 
\end{center}

Consider the map $g_{m+1}$ defined on $T_m$ by:

\[
\forall k \leq m, \ \left\{ 
\begin{array}{l} 
g_{m+1} (x _{\sigma (k)}) = d^{\m{X}}(z _{\tau (m+1)} , \varphi ( x _{\sigma (k)}))\\ 
g_{m+1} (\varphi ^{-1} (z _{\tau (k)})) = d^{\m{X}}(z_{\tau (m+1)} , z _{\tau (k)} )
\end{array} \right.
\] 

Then $g_{n+1}$ is Kat\v{e}tov over $T_m$ and the union $T_m \cup \{ g_{m+1}\}$ is isometric to $ \varphi (T_m \cup \{ z_{\tau (m+1)} \})$. So again, by hypothesis on $\m{X}$, we can find a point in $O(g_{m+1},\m{X})$, call it $\varphi ^{-1} (z_{\tau (m+1)})$. After infinitely many steps, we are left with an isometry $\varphi$ with domain $ \m{X} = \{ x_k : k \in \N \}$ and range $\m{X} = \{ z_k : k \in \N \}$. This finishes the proof.

\subsection{Proof of Lemma \ref{lem:red}}

\label{subsection:red}

Lemma \ref{lem:red} is a modified version of a result proved in \cite{NVTS}, namely Lemma 5. Let $G_0 \subset G$ be finite subsets of $\Ur_{\N}$, $\mathcal{G}$ a family of Kat\v{e}tov maps with domain $G$ and such that for every $g, g' \in \mathcal{G}$: \[ \max(\restrict{|g - g'|}{G_0}) = \max | g - g'|, \] \[ \min(\restrict{(g+g')}{G_0}) = \min(g + g').\] 

We need to produce an isometric copy $\m{C}$ of $\Ur _{\N}$ inside $\Ur _{\N}$ such that:
\begin{enumerate} 
\item $G \cap C = G_0$.
\item $\forall g \in \mathcal{G} \ \ O(\restrict{g}{G_0} , \m{C}) \subset O(g, \Ur _{\N}).$  
\end{enumerate} 

First, observe that it suffices to provide the proof assuming that $G$ is of the form $G_0\cup\{ z\}$. The general case is then handled by repeating the procedure. 
  
\begin{lemma}

\label{lem:red1}

Let $\m{X}$ be a finite subspace of $\bigcup \{ O(\restrict{g}{G_0}, \Ur_{\N}) : g \in \mathcal{G}\}$. Then there is an isometry $\varphi$ on $\Ur _{\N}$ fixing $G_0 \cup (X \cap \bigcup \{ O(g, \Ur_{\N}) : g \in \mathcal{G}\})$ and such that: \[ \forall g \in \mathcal{G} \ \ \varphi \left( X \cap O(\restrict{g}{G_0}, \Ur_{\N})\right) \subset O(g, \Ur_{\N}).\]

\end{lemma}

\begin{proof}
For $x \in X$, there is a unique element $g_x \in \mathcal{G}$ such that $x \in O(\restrict{g_x}{G_0}, \Ur_{\N})$. Let $k$ be the map defined on $G_0 \cup X$ by
\begin{displaymath}
k(x) = \left \{ \begin{array}{cl}
 d^{\Ur _{\N}}(x,z) & \textrm{if $x \in G_0$,} \\
 g_x(z) & \textrm{if $x \in X$.}
 \end{array} \right.
\end{displaymath}

\begin{claimm}
The map $k$ is Kat\v{e}tov. 
\end{claimm}

\begin{proof}
The metric space $G_0\cup \{z\}$ witnesses that $k$ is Kat\v{e}tov over $G_0$. Hence, it suffices to check that for every $x \in X$ and $y \in G_0\cup X$, \[ |k(x) - k(y)| \leq d^{\Ur _{\N}} (x,y) \leq k(x) + k(y).\]

Consider first the case $y \in G_0$. Then $d^{\Ur}(x,y) = g_x(y)$ and we need to check that \[ |g_x(z) - d^{\Ur _{\N}}(y,z)| \leq g_x(y) \leq g_x(z) + d^{\Ur _{\N}}(y,z).\] 

Or equivalently, \[ |g_x(z) - g_x(y)| \leq d^{\Ur _{\N}}(y,z) \leq g_x(z) + g_x(y).\] 

But this is true since $g_x$ is Kat\v{e}tov over $G_0\cup \{z\}$. Consider now the case $y\in X$. Then $k(y) = g_y(z)$ and we need to check \[ |g_x(z) - g_y(z)| \leq d^{\Ur _{\N}}(x,y) \leq g_x(z) + g_y(z).\] 

But since $\m{X}$ is a subspace of $\bigcup \{ O(\restrict{g}{G_0}, \Ur_{\N}) : g \in \mathcal{G}\}$, we have, for every $u \in G_0$, \[ |d^{\Ur _{\N}}(x,u) - d^{\Ur _{\N}}(u,y)| \leq d^{\Ur _{\N}}(x,y) \leq d^{\Ur _{\N}}(x,u) + d^{\Ur _{\N}}(x,u).\]

Since $x \in O(\restrict{g_x}{G_0}, \Ur_{\N})$ and $y \in O(\restrict{g_y}{G_0}, \Ur_{\N})$, this is equivalent to \[ |g_x(u) - g_y(u)| \leq d^{\Ur _{\N}}(x,y) \leq g_x(u) + g_y(u).\] 

Therefore, \[ \max (\restrict{|g_x - g_y|}{G_0}) \leq d^{\Ur _{\N}}(x,y) \leq \min (\restrict{(g_x + g_y)}{G_0}) .\] 

Now, by hypothesis on $\mathcal{G}$, this latter inequality remains valid if $G_0$ is replaced by $G_0\cup \{ z\}$. The required inequality follows. \end{proof}

By ultrahomogeneity of $\Ur _{\N}$ (or, more precisely, by its equivalent reformulation provided in Lemma \ref{prop:extension}), we can pick a point $z' \in O(k, \Ur_{\N})$. The metric space $G_0\cup(\m{X}\cap\bigcup \{ O(g, \Ur_{\N}) : g \in \mathcal{G}\})\cup \{k\}$ being isometric to the subspace of $\Ur _{\N}$ supported by $G_0\cup(X\cap\bigcup \{ O(g, \Ur_{\N}) : g \in \mathcal{G}\})\cup \{z\}$, so is the subspace of $\Ur _{\N}$ supported by $G_0\cup(X\cap\bigcup \{ O(g, \Ur_{\N}) : g \in \mathcal{G}\})\cup \{z'\}$. By ultrahomogeneity again, we can find a surjective isometry $\varphi$ of $\Ur _{\N}$ fixing $G_0\cup(X\cap\bigcup \{ O(g, \Ur_{\N}) : g \in \mathcal{G}\})$ and such that $\varphi(z')=z$. Then $\varphi$ is as required: let $g \in \mathcal{G}$ and $x \in O(\restrict{g}{G_0}, \Ur_{\N})$. Then: \[d^{\Ur _{\N}}(\varphi(x),z) = d^{\Ur _{\N}}(\varphi(x), \varphi(z')) = d^{\Ur _{\N}}(x,z') = k(x) = g(z). \] 

That is, $\varphi(x) \in O(g, \Ur_{\N})$. \end{proof}

\begin{lemma}

\label{lem:red2}

There is an isometric embedding $\psi$ of $G_0 \cup \bigcup \{ O(\restrict{g}{G_0}, \Ur_{\N}) : g \in \mathcal{G})\}$ into $G_0 \cup \bigcup \{ O(g, \Ur_{\N}) : g \in \mathcal{G})\}$ fixing $G_0$ such that: \[ \forall g \in \mathcal{G} \ \ \psi \left( O(\restrict{g}{G_0}, \Ur_{\N})\right) \subset O(g, \Ur_{\N}).\] 

\end{lemma}

\begin{proof}
Let $\{ x_n : n \in \N\}$ enumerate $\bigcup \{ O(\restrict{g}{G_0}, \Ur_{\N}) : g \in \mathcal{G})\}$. For $n \in \N$, let $g_n$ be the only $g \in \mathcal{G}$ such that $x_n \in O(\restrict{g_n}{G_0}, \Ur_{\N})$. Apply Lemma \ref{lem:red1} inductively to construct a sequence $(\psi _n)_{n \in \N}$ of surjective isometries of $\Ur _{\N}$ such that for every $n\in \N$, $\psi_n$ fixes $G_0 \cup \psi_{n-1}\left(\{x_k : k < n \}\right)$ and $\psi_n(x_n) \in O(g_n, \Ur_{\N})$. Then $\psi$ defined on $G_0 \cup \{ x_n : n \in \N\}$ by $\restrict{\psi}{G_0} = id_{G_0}$ and $\psi(x_n) = \psi_n(x_n)$ is as required. \end{proof}

We now turn to the proof of Lemma \ref{lem:red}. Let $\m{Y}$ and $\m{Z}$ be the metric subspaces of $\Ur _{\N}$ supported by $G \cup \bigcup \{ O(g, \Ur_{\N}) : g \in \mathcal{G})\}$ and $G_0 \cup \bigcup \{ O(\restrict{g}{G_0}, \Ur_{\N}) : g \in \mathcal{G})\}$ respectively. Let $i_0 : \funct{\m{Z}}{\Ur _{\N}}$ be the isometric embedding provided by the identity. By Lemma~\ref{lem:red2}, the space $\m{Z}$ embeds isometrically into $\m{Y}$ via an isometry $j_0$ that fixes $G_0$. We can therefore consider the metric space $\m{W}$ obtained by gluing $\Ur _{\N}$ and $\m{Y}$ via an identification of $\m{Z} \subset \Ur _{\N}$ and $j_0\left(\m{Z}\right) \subset \m{Y}$. The space $\m{W}$ is described in Figure 1. 

Formally, the space $\m{W}$ can be constructed thanks to a property of  countable metric spaces with distances in $\N$ known as \emph{strong amalgamation}: we can find a countable metric space $\m{W}$ with distances in $\N$ and isometric embeddings $i_1 : \funct{\Ur _{\N}}{\m{W}}$ and $j_1 : \funct{\m{Y}}{\m{W}}$ such that: 
\begin{itemize}
	\item $i_1 \circ i_0 = j_1 \circ j_0$,
	\item $W = i_1\left(U_{\N}\right) \cup j_1\left(Y\right)$,
	\item $i_1\left(U_{\N}\right) \cap j_1\left(Y\right) = (i_1 \circ i_0)\left(Z\right) = (j_1 \circ j_0)\left(Z\right)$, 
	\item for every $x \in U_{\N}$ and $y \in Y$: \begin{align*}d^{\m{W}}(i_1(x),j_1(y)) & = \min \{ d^{\m{W}}(i_1(x),i_1 \circ i_0 (z)) + d^{\m{W}}(j_1 \circ j_0 (z),j_1(y)) : z \in Z\} \\
& = \min \{ d^{\Ur _{\N}}(x,i_0 (z)) + d^{\m{Y}}(j_0 (z),y) : z \in Z\} \\
& = \min \{ d^{\Ur _{\N}}(x,z) + d^{\m{Y}}(j_0 (z),y) : z \in Z\}.
\end{align*}
\end{itemize}
\vskip-5pt
\begin{figure}[h]
\begin{center}
\hskip-10pt\includegraphics[width=132.00mm]{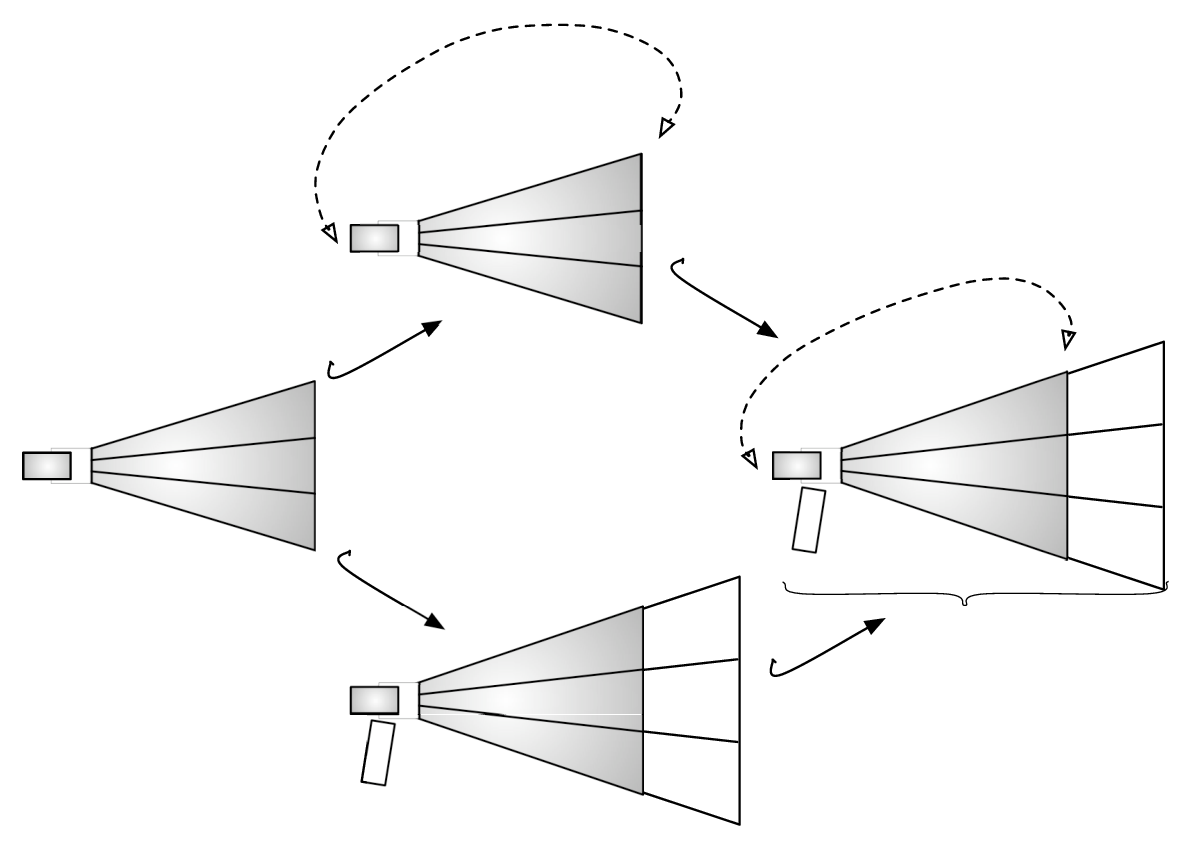}%
\drawat{-75.26mm}{93.04mm}{$\Ur _{\N}$}%
\drawat{-128.91mm}{44.55mm}{$G_0$}%
\drawat{-114.80mm}{46.80mm}{$O(\restrict{g_1}{G_0}, \Ur _{\N})$}%
\drawat{-114.80mm}{41.mm}{$O(\restrict{g_2}{G_0}, \Ur _{\N})$}%
\drawat{-114.80mm}{36.mm}{$O(\restrict{g_3}{G_0}, \Ur _{\N})$}%
\drawat{-93.35mm}{26.26mm}{$j_0$}%
\drawat{-93mm}{19mm}{$G_0$}%
\drawat{-95.19mm}{10mm}{$G$}%
\drawat{-77.50mm}{22.41mm}{$O(g_1, \Ur _{\N})$}%
\drawat{-77.50mm}{14mm}{$O(g_2, \Ur _{\N})$}%
\drawat{-77.50mm}{7mm}{$O(g_3, \Ur _{\N})$}%
\drawat{-54.64mm}{57.79mm}{$i_1$}%
\drawat{-91.30mm}{50.69mm}{$i_0$}%
\drawat{-24.49mm}{64.92mm}{$ i_1 \left(\Ur _{\N}\right)$}%
\drawat{-27.5mm}{23mm}{$\m{W}$}%
\drawat{-40.80mm}{19.mm}{$j_1$}%
\drawat{-50.80mm}{36mm}{$j_1\left(G\right)$}%
\end{center}
\caption{The space $\m{W}$}
\end{figure}

The crucial point here is that in $\m{W}$, every $x \in i_1 \left( U_{\N}\right)$ realizing some $\restrict{g}{G_0}$ over $i_1 \left( G_0\right)$ also realizes $g$ over $j_1 \left( G\right)$.  

Using $\m{W}$, we show how $\m{C}$ can be constructed inductively: consider an enumeration $\{ x_n : n \in \N\}$ of $i_1\left(U_{\N}\right)$ admitting $i_1\left(G_0\right)$ as an initial segment. Assume that the points $\varphi (x_0),\ldots , \varphi(x_n)$ are constructed so that: 
\begin{itemize}
	\item the map $\varphi$ is an isometry, 
	\item $\dom \varphi \subset i_1\left(U_{\N}\right)$,
	\item $\varphi (x_0),\ldots , \varphi(x_n) \in U_{\N}$,
	\item $\varphi(i_1(x)) = x$ whenever $x \in G_0$,
	\item $d^{\Ur _{\N}} (\varphi (x_k) , z) = d^{\m{W}} (x_k , j_1(z))$ whenever $z \in G$ and $k \leq n$.
\end{itemize}

We want to construct $\varphi (x _{n+1})$. Consider $e$ defined on $\{ \varphi (x_k) : k \leq n\} \cup G $ by: 
\begin{displaymath}
\left \{ \begin{array}{l}
 \forall k \leq n \ \ e(\varphi(x_k)) = d^{\m{W}} (x_k , x_{n+1}), \\
 \forall z \in G \ \ e(z) = d^{\m{W}} (j_1(z) , x_{n+1}).
 \end{array} \right.
\end{displaymath}

Observe that the metric subspace of $\m{W}$ given by $\{x_k : k \leq n+1\} \cup j_1\left(G\right) $ witnesses that $e$ is Kat\v{e}tov. It follows that the set $O(e, \Ur_{\N})$ is not empty and $\varphi (x_{n+1})$ can be chosen in it. $\qed$

\section{Proof of Theorem \ref{thm:U}} \label{section:U}

The purpose of this section is to prove Theorem \ref{thm:U}. So let $U = B \cup R$ and $\varepsilon > 0$. Assume that there is a compact metric subspace $\m{K}$ of $\Ur$ that does not embed in $(B)_{\varepsilon}$. We wish to show that $\Ur$ embeds in $(R)_{\varepsilon}$. 

\subsection{Proof of Theorem \ref{thm:U}}

We will use the result of Theorem \ref{thm:Uomega} as well as two technical lemmas, whose proofs are postponed to sections \ref{subsection:lemrich'} and \ref{subsection:lemrich''}. 

\begin{lemma}
\label{lem:rich'}
Let $q \in \N$ be positive. Then there is an isometric copy $\Ur_{\N/q} ^*$ of $\Ur_{\N/q}$ in $\Ur$ such that for every subspace $\mc{V}$ of $\Ur_{\N/q} ^* $ isometric
to $\Ur_{\N/q}$, the set $(\tilde{V})_{1/q}$ includes an isometric copy of $\Ur$.
\end{lemma}


The second lemma states that in $\Ur$, the copies of the compact space $\m{K}$ can be captured by a single finite metric subspace of $\Ur$: 

\begin{lemma}
\label{lem:rich''}
There is a finite metric subspace $\m{Y}$ of $\Ur$ with rational distances such that $\m{K}$ embeds in $(\tilde{Y})_{\varepsilon/2}$ for every subspace $\mc{Y}$ of $\Ur$ isometric to $\m{Y}$.
\end{lemma}

%
%

Assuming Lemma \ref{lem:rich'} and Lemma \ref{lem:rich''}, the proof of Theorem \ref{thm:U} goes as follows: choose $q \in \N$ large enough so that $1/q\leq \varepsilon$ and all distances appearing in $\m{Y}$ are integer multiples of $1/q$. The partition $U = B\cup R$ induces a partition of $U_{\N/q} ^*$ provided by Lemma \ref{lem:rich'}. Note that $\m{Y}$ does not embed in $B$: indeed, if a subspace $\mc{Y}$ of $B$ were isometric to $\m{Y}$, then $(\tilde{Y})_{\varepsilon} \subset (B)_{\varepsilon}$ and by Lemma \ref{lem:rich''}, the space $\m{K}$ would embed in $(B)_{\varepsilon}$, which is not the case. Observe now that by weak indivisibility of the space $\Ur_{\N}$ (Theorem \ref{thm:Uomega}), the space $\Ur_{\N/q}$ is weakly indivisible as well, so there is a subspace $\mc{V}$ of $\Ur_{\N/q} ^* $ isometric to $\Ur_{\N/q}$ such that $\tilde{V} \subset R$. By construction of $\Ur_{\N/q} ^*$, the set $(\tilde{V})_{1/q}$ includes an isometric copy $\mc{U}$ of $\Ur$. To complete the proof, notice that $\tilde{U} \subset (\tilde{V})_{1/q} \subset (\tilde{V})_{\varepsilon} \subset (R)_{\varepsilon}$.

\subsection{Proof Lemma \ref{lem:rich'}}

\label{subsection:lemrich'}

Lemma \ref{lem:rich'} is a modified version of a result proved in \cite{LANVT}, whose statement appears at the very beginning of Proposition 5. Its proof is an easy modification of Lemma 2 \cite{LANVT} and is included here for completeness. The core of the proof is contained in Lemma
\ref{lem:hedgehog} which we present now. Fix an enumeration $\{y_n : n \in \N \}$ of $U
_{\Q}$. For a number $\alpha$, let $\lceil \alpha \rceil _q$ denote the smallest number $\geq \alpha$ of the form $l/q$ with $l$ integer. The function $\lceil \cdot \rceil_q$ is subadditive and non decreasing. Hence, the composition $\lceil d^{\m{Z}}\rceil_q =\lceil \cdot \rceil_q \circ d^{\Ur _{\Q}}$ is a metric on $U_{\Q}$. Let $\m{X} _q$ be the metric space $(U _{\Q} , \left\lceil d^{\Ur _{\Q}} \right\rceil _q)$.
The underlying set of $\m{X} _q$ is really $\{y_n : n \in \N \}$ but to avoid confusion, we
write it $\{x_n : n \in \N \}$, being understood that for every $n \in \N$, $x_n = y_n$. On
the other hand, note that $\Ur _{\N/q}$ and $\m{X} _q$ embed isometrically into each other: $\m{X} _q$ embeds in $\Ur _{\N/q}$ because any countable metric space with distances in $\N/q$ embeds in $\Ur_{\N/q}$, and $\Ur_{\N/q}$ embeds in $\m{X}_q$ because any copy of $\Ur_{\N/q}$ in $\Ur_{\Q}$ remains isometric to $\Ur_{\N/q}$ in $\m{X}_q = (U _{\Q} , \left\lceil d^{\Ur _{\Q}} \right\rceil _q)$.

\begin{lemma}
\label{lem:hedgehog} There is a countable metric space $\m{Z}$ including $\m{X} _q$ such that for every strictly increasing $\sigma : \funct{\N}{\N}$
 such that  $x_n \mapsto x_{\sigma (n)}$ is an isometry, the set  $ (\{ x_{\sigma (n)} : n \in \N \})_{1/q}$
 includes an isometric copy of $\Ur _{\Q}$.
\end{lemma}
Assuming Lemma \ref{lem:hedgehog}, we now show how we can construct $\Ur ^* _{\N/q}$. The space $\m{Z}$ is
countable so we may assume that it is a subspace of $\Ur$. Now, take $\Ur
_{\N/q} ^*$ a subspace of $\m{X}_q$ and isometric to $\Ur _{\N/q}$. We claim that $\Ur _{\N/q} ^* $ works: let
$\mc{V}$ be a subspace of $\Ur _{\N/q} ^* $ isometric to $\Ur _{\N/q}$. We first show that $(\tilde{V})_{1/q}$ includes a copy of $\Ur _{\Q}$. The enumeration $\{x_n : n \in \N \}$ induces a linear ordering $<$ of $\tilde{V}$. According to Lemma \ref{lem:hedgehog}, it suffices to show that $(\mc{V},<)$ includes a copy of $\{x_n : n \in \N \}_<$ seen as an ordered metric space. To do that, observe that since $\m{X} _q$ embeds isometrically into $\Ur _{\N/q}$, there is a
linear ordering $<^*$ of $\Ur _{\N/q}$ such that $\{x_n : n \in \N \}_<$ embeds into
$(\Ur _{\N/q} , <^*)$ as ordered metric space. Therefore, it is enough to show:
\begin{claimm}
$(\mc{V},<)$ includes a copy of $(\Ur _{\N/q} , <^*)$.
\end{claimm}
\begin{proof}
Write
\begin{align*}
(U_{\N/q} , <^*)& = \{s_n : n \in \N \}_{<^*}\\
(\tilde{V},<) &  = \{t_n : n \in \N \}_{<}.
\end{align*}

Let $\sigma (0) = 0$. If $\sigma (0) < \dots < \sigma (n)$ are chosen such that $s_k \mapsto
t_{\sigma (k)}$ is a finite isometry, observe that the following set is infinite
\[\{ i \in \N : \forall k \leqslant n \ \ d^{\Ur _{\N/q}}(t_{\sigma (k)} , t_i) = d^{\Ur _{\N/q}}(s_k ,
s_{n+1})\}.\]

Therefore, simply take $\sigma (n+1) $ in that set and larger than $\sigma (n)$.
\end{proof} Observe that since the metric completion of $\Ur _{\Q}$ is $\Ur$,
the closure of $(\tilde{V})_{1/q}$ in $\Ur$ includes a  copy of $\Ur$. Hence we are done since
$(\tilde{V})_{1/q}$ is closed in $\Ur$.

We now turn to the proof of Lemma \ref{lem:hedgehog}. The strategy is first to provide the set $Z$
where the required metric space $\m{Z}$ is supposed to be based on, and then to argue that the
distance $d^{\m{Z}}$ can be obtained (Lemmas \ref{lem:3} to \ref{lem:6}). To construct $Z$, proceed
as follows: for $t \subset \N$, write $t$ as the strictly increasing enumeration of its
elements:

\begin{center}
$t = \{t_i : i \in |t| \}_<$.
\end{center}

Now, let $T$ be the set of all finite nonempty subsets $t$ of $\N$ such that $x_n \mapsto
x_{t_n}$ is an isometry between $\{ x_n : n \in |t| \}$ and $\{x_{t_n} : n \in |t|\}$. This
set $T$ is a tree (in the order-theoretic sense) when ordered by end-extension. Let

\begin{center}
$Z = X_q \overset{.}{\cup} T$.
\end{center}

For $z \in Z$, define
\begin{displaymath}
\pi(z) = \left \{ \begin{array}{cl}
 z & \textrm{if $z \in X_q$.} \\
 x_{\max z } & \textrm{if $z \in T$.}
 \end{array} \right.
\end{displaymath}

Now, consider an edge-labelled graph structure on $Z$ by defining   $\delta$ with domain
$\mathrm{dom} (\delta) \subset Z \times Z $ as follows:
\begin{itemize}
\item If $s, t \in T$, then $(s,t) \in \mathrm{dom}(\delta)$ iff $s$ and $t$ are $<_T$ comparable. In this case, \[\delta (s,t) = d^{\Ur _{\Q}} (y_{|s|-1 }, y_{|t|-1}).\]

\item If $x, y \in X_q$, then $(x,y)$ is always in $\mathrm{dom}(\delta)$ and
\[\delta (x,y) = d^{\m{X}_q} (x, y).\]

\item If $t \in T$ and $x \in X_q$, then $(x,s)$ and $(s,x)$ are in $\mathrm{dom}(\delta)$ iff $x = \pi (t)$. In this case
\[\delta (x,s) = \delta (s,x) = \frac1{q}.\]
\end{itemize}

For a branch $b$ of $T$ and $i \in \N$, let $b(i)$ be the unique element of $b$ with height $i$
in $T$. Observe that $b(i)$ is a $(i+1)$-element subset of $\N$. Observe also that for every $i,j \in \N$,
$b(i)$ is connected to $\pi (b(i))$ and $b(j)$, and
\begin{enumerate}
\item $\delta (b(i), \pi (b(i)) = 1/{q}$,
\item $\delta (b(i), b(j)) = d^{\Ur _{\Q}} (y_i , y_j)$,
\item $\delta(\pi(b(i)),\pi(b(j)))$ is equal to any of the following quantities: \[d^{\m{X}_q}(x_{\max b(i)},x_{\max b(j)})= d^{\m{X}_q}(x_i,x_j)=  \lceil d^{\Ur_Q}(y_i,y_j) \rceil_q.\]
\end{enumerate}

In particular, if $b$ is a branch of $T$, then $\delta$ induces a metric on $b$ and the map from
$\Ur _{\Q}$ to $b$ mapping  $y_i$ to $b(i)$ is a surjective isometry. We claim that if we can show
that $\delta$ can be extended to a metric $d^{\m{Z}}$ on $Z$, then Lemma
\ref{lem:hedgehog} will be proved. Indeed, let
\[\tilde{X} _q = \{ x_{\sigma (n)} : n \in \N \} \subset X _q,\] with $\sigma : \funct{\N}{\N}$ strictly increasing and
$x_n \mapsto x_{\sigma (n)}$ distance preserving. See $\mathrm{ran}(\sigma)$, the range of $\sigma$, as a branch $b$ of $T$. Then
$(b, d^{\m{Z}}) = (b, \delta)$ is isometric to $\Ur _{\Q}$ and
\[b \subset (\pi (b) )_{ 1/{q}} = (\tilde{X} _q)_{1/{q}} .\]

Our goal now is consequently to show that $\delta$ can be extended to a metric on $Z$ with values
in $[0,1]$. For $x, y \in Z$, and $n \in \N$ strictly positive, define \emph{a path from $x$ to
$y$ of size $n$} as a finite sequence $\gamma = (z_i)_{i<n}$ such that $z_0 = x$, $z_{n-1} = y$
and for every $i<n-1$,
\[(z_i, z_{i+1}) \in \dom(\delta).\]

For $x, y$ in $Z$, let $P(x,y)$ be the set of all paths from $x$ to $y$. If $\gamma = (z_i)_{i<n}$ is
in $P(x,y)$, $ \| \gamma \|$ is defined as:
\[ \| \gamma \| = \sum _{i=0} ^{n-1} \delta (z_i , z_{i+1} ).\]

We are going to see that the required metric can be obtained with
$d^{\m{Z}}$ defined by
\[ d^{\m{Z}}(x,y) = \inf \{ \| \gamma \| : \gamma \in P(x,y)\} .\]

Equivalently, we are going to show that for every  $(x,y) \in \dom (\delta)$, every path $\gamma$
from $x$ to $y$ is \emph{metric}, i.e. satisfies the following inequality: 
\begin{equation}
\label{hojthurhgr}\delta (x,y) \leqslant \| \gamma \|. 
\end{equation}

Let $x, y \in Z$. Call a path $\gamma$ from $x$ to $y$ \emph{trivial} when $\gamma = (x,y)$ and
\emph{irreducible} when no proper subsequence of $\gamma$ is a non-trivial path from $x$ to $y$.
Finally, say that $\gamma$ is a \emph{cycle} when $(x,y) \in \dom (\delta)$. It should be clear
that to prove that $d^{\m{Z}}$ works, it is enough to show that the previous inequality
\eqref{hojthurhgr} is true for every irreducible cycle. Note that even though $\delta$ takes only
rational values, it might not be the case for $d^{\m{Z}}$. We now turn to the study of the
irreducible cycles in $Z$.

\begin{lemma}

\label{lem:3}

Let $x, y \in T$. Assume that $x$ and $y$ are not $<_T$-comparable. Let $\gamma$ be an irreducible
path from $x$ to $y$ in $T$. Then there is $z \in T$ such that $z <_T x$, $z <_T y$ and $\gamma =
(x,z,y)$.

\end{lemma}

\begin{proof}

Write $\gamma = (z_i)_{i<n+1}$. $z_1$ is connected to $x$ so $z_1$ is $<_T$-comparable with $x$.
We claim that $z_1 <_T x$ : otherwise, $x <_T z_1$ and every element of $T$ which is
$<_T$-comparable with $z_1$ is also $<_T$-comparable with $x$. In particular, $z_2$ is
$<_T$-comparable with $x$, a contradiction since $z_2$ and $x$ are not connected. We now claim that
$z_1 <_T y$. Indeed, observe that $z_1 <_T z_2$ : otherwise, $z_2 <_T z_1 <_T x$ so $z_2 <_T x$
contradicting irreducibility. Now, every element of $T$ which is $<_T$-comparable with $z_2$ is
also $<_T$-comparable with $z_1$, so no further element can be added to the path. Hence $z_2 = y$
and we can take $z_1 = z$. \end{proof}

\begin{lemma}

\label{lem:4}

Every non-trivial irreducible cycle in $X_q$ has size $3$.

\end{lemma}

\begin{proof}

Obvious since $\delta $ induces the metric $d^{\m{X}_q}$ on $X_q$. \end{proof}

\begin{lemma}

\label{lem:4'}

Every non-trivial irreducible cycle in $T$ has size $3$ and is included in a branch.

\end{lemma}

\begin{proof}

Let $c = (z_i)_{i<n}$ be a non-trivial irreducible cycle in $T$. We may assume that $z_0 <_T
z_{n-1}$. Now,  observe that every element of $T$ comparable with $z_0$ is also comparable with
$z_{n-1}$. In particular, $z_1$ is such an element. It follows that $n = 3$ and that $z_0, z_1,
z_2$ are in a same branch. \end{proof}

\begin{lemma}

\label{lem:5}

Every irreducible cycle in $Z$ intersecting both $X_q$ and $T$ is supported by a set whose form is one of the following ones.
\begin{center}
\begin{figure}[h]
\includegraphics[scale=0.73]{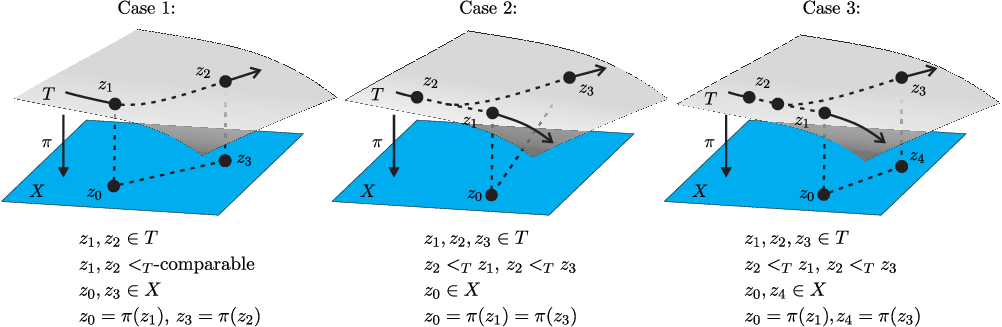}
\caption{Irreducible cycles}\label{figu1}
\end{figure}
\end{center}
%

\end{lemma}

\begin{proof}

Let $C$ be a set supporting an irreducible cycle $c$ intersecting both $X_q$ and $T$. It should be
clear that $|C \cap X_q| \leqslant 2$: otherwise since any two points in $X_q$ are connected, $c$ would admit a strict subcycle, contradicting irreducibility.

If $C \cap X_q$ has size $1$, let $z_0$ be its unique element. In $c$, $z_0$ is connected to two
elements which we denote $z_1$ and $z_3$. Note that $z_1, z_3 \in T$ so $\pi (z_1) = \pi (z_3) =
z_0$. Since elements in $T$ which are connected never project on a same point, it follows that
$z_1, z_3$ are $<_T$-incomparable. Now, $c$ induces an irreducible path from $z_1$ to $z_3$ in $T$
so from Lemma \ref{lem:3}, there is $z_2 \in C$ such that $z_2 <_T z_1$, $z_2 <_T z_3$, and we are
in case 2.

Assume now that $C \cap X_q = \{ z_0 , z_4 \}$. Then there are $z_1, z_3 \in C \cap T$ such that
$\pi(z_1) = z_0$  and $\pi(z_3) = z_4$. Note that since $z_0 \neq z_4$, we must have $z_1 \neq
z_3$. Now, $C \cap T$ induces an irreducible path from $z_1$ to $z_3$ in $T$. By Lemma \ref{lem:3},
either $z_1$ and $z_3$ are compatible and in this case, we are in case 1, or $z_1$ and $z_3$ are
$<_T$-incomparable and there is $z_2$ in $C\cap T$ such that $z_2 <_T z_1$, $z_2 <_T z_3$ and we
are in case 3. \end{proof}

\begin{lemma}

\label{lem:6}

Every non-trivial irreducible cycle in $Z$ is metric.

\end{lemma}

\begin{proof}
Let $c$ be an irreducible cycle in $Z$. If $c$ is supported by $X_q$, then, by Lemma \ref{lem:4}, $c$
has size $3$ and  is metric since $\delta$ induces a metric on $X_q$. If $c$ is supported by $T$,
then, by Lemma \ref{lem:4'}, $c$ also has size $3$ and is included in a branch $b$ of $T$. Since
$\delta$ induces a metric on $b$, $c$ is metric. We consequently assume that $c$ intersects both
$X_q$ and $T$. According to Lemma \ref{lem:5}, $c$ is supported by a set $C$ whose form is covered
by one of the cases 1, 2 or 3. So to prove the present lemma, it is enough to show every cycle
obtained from a re-indexing of the cycles described in those cases is metric.

Case 1: The required inequalities are obvious after having observed that \[\delta (z_0 , z_3) = \left\lceil \delta (z_1, z_2) \right\rceil _q  \text{ and }  \delta
(z_0 , z_1) = \delta (z_2 , z_3) = \frac1{q}.\]

Case 2: Notice that $\delta (z_0 , z_1) = \delta (z_0 , z_3) = 1/q$. So the inequalities we need to prove are \begin{align}
\delta (z_1 , z_2) & \leqslant   \delta (z_2, z_3) + \frac2q, \label{lulu1}\\
\delta (z_2 , z_3) & \leqslant  \delta (z_1, z_2) + \frac2q. \label{lulu2}
\end{align}

By symmetry, it suffices to verify that \eqref{lulu1} holds. Observe that since $\pi (z_1) = \pi (z_3) = z_0$, we must have $\left\lceil \delta (z_1 , z_2) \right\rceil_q=\left\lceil \delta(z_2 , z_3)\right\rceil_q$. So: \[\delta (z_1 , z_2) \leqslant  \left\lceil \delta (z_1 , z_2) \right\rceil _q = \left\lceil \delta (z_2 , z_3) \right\rceil _q \leqslant  \delta (z_2 , z_3) + \frac2q.\]

Case 3: Observe that $\delta (z_0 , z_1) = \delta (z_3 , z_4) = 1/q$, so the inequalities we need
to prove are
\begin{align}
\delta (z_1 , z_2) & \leqslant   \delta (z_2, z_3) + \delta (z_0 , z_4) + \frac2q, \label{ohrtjuer1}\\
\delta (z_0 , z_4) & \leqslant  \delta (z_1, z_2) + \delta (z_2 , z_3) + \frac2q. \label{ohrtjuer2}
\end{align}

For \eqref{ohrtjuer1}:
\begin{align*}
\delta (z_1 , z_2) & \leqslant  \left\lceil \delta (z_1 , z_2) \right\rceil _q \\
& =  \delta (\pi (z_1) , \pi (z_2) ) \\
& =  \delta (z_0 , \pi (z_2) ) \\
& \leqslant  \delta (z_0 , z_4) + \delta (z_4 , \pi (z_2)) \\
& =  \delta (z_0 , z_4) + \left\lceil \delta (z_3 , z_2) \right\rceil _q \\
& \leqslant  \delta (z_0 , z_4) + \delta (z_2 , z_3) + \frac2q.
\end{align*}

For \eqref{ohrtjuer2}: Write $z_1 = b(j)$, $z_3 = b'(k)$, $z_2 = b(i) = b'(i)$. Then $z_0 = \pi
(z_1) = x_{\max b(j)}$ and $z_4 = \pi (z_3) = x_{\max b'(k)}$. Observe also that $\delta (z_1 ,
z_2) = d^{\Ur _{\Q}}(y_j , y_i)$ and that $\delta (z_2 , z_3) = d^{\Ur _{\Q}}(y_i , y_k)$. So:
\begin{align*}
\delta (z_0 , z_4) & =  d^{\m{X}_q}(x_{\max b(j)}, x_{\max b'(k)})\\
& \leqslant  d^{\m{X}_q}(x_{\max b(j)}, x_{\max b(i)}) + d^{\m{X}_q}(x_{\max b'(i)}, x_{\max b'(k)}) \\
& = d^{\m{X}_q}(x_j, x_i) + d^{\m{X}_q}(x_i, x_k) \\
& =  \left\lceil d^{\Ur _{\Q}}(y_j, y_i)\right\rceil _q + \left\lceil d^{\Ur _{\Q}}(y_i, y_k) \right\rceil _q \\
& = \left\lceil \delta (z_1, z_2) \right\rceil _q + \left\lceil \delta (z_2 , z_3) \right\rceil _q \\
& \leqslant  \delta (z_1, z_2) + \frac1q + \delta (z_2 , z_3) + \frac1q \\
& = \delta (z_1, z_2) + \delta (z_2 , z_3) + \frac2q. \qedhere
\end{align*}
\end{proof}

\subsection{Proof of Lemma \ref{lem:rich''}}

\label{subsection:lemrich''}

Using compactness of $\m{K}$, find a finite subspace $\m{Z}$ of $\m{K}$ such that $K \subset (Z)_{\varepsilon/2}$. 

\begin{claimm}
\label{claim:K}
The space $\m{K}$ embeds in $(\tilde{Z})_{\varepsilon}$ for every subspace $\mc{Z}$ of $\Ur$ isometric to $\m{Z}$. 
\end{claimm}

\begin{proof}
This follows from ultrahomogeneity of $\Ur$: if $\mc{Z}$ is a subspace of $\Ur$ isometric to $\m{Z}$, let $\phi : \funct{Z}{\tilde{Z}}$ be an isometry. By ultrahomogeneity of $\Ur$, find $\Phi:\funct{U}{U}$ extending $\phi$. Then $\Phi(K)$ is isometric to $\m{K}$ and is included in \[ \Phi((Z)_{\varepsilon/2}) = (\Phi(Z))_{\varepsilon/2} = (\tilde{Z})_{\varepsilon/2}. \qedhere\] 
\end{proof}

Therefore, the space $\m{Z}$ is almost as required except that it may not have rational distances. To arrange that, consider $q \in \N$ large enough so that $1/q < \varepsilon /2$. Recall that for a number $\alpha$, $\lceil \alpha \rceil _q$ denotes the smallest number $\geq \alpha$ of the form $l/q$ with $l$ integer. The function $\lceil \cdot \rceil_q$ is subadditive and non decreasing. Hence, the composition $\lceil d^{\m{Z}}\rceil_q =\lceil \cdot \rceil_q \circ d^{\m{Z}}$ is a metric on $Z$. Let $\m{Y}$ be defined as the metric space $(Z,\lceil d^{\m{Z}}\rceil_q)$. It obviously has rational distances. We are going to show that it is as required. Consider the set $X = Z \times \{ 0, 1\}$ and define

\begin{displaymath}
\delta((z, i),(z',i')) = \left \{ \begin{array}{ll}
 d^{\m{Z}}(z,z')  & \textrm{if $i=i'=0$,} \\
 \lceil d^{\m{Z}} (z,z')\rceil_q  & \textrm{if $i=i'=1$,} \\
 d^{\m{Z}}(z,z') + \varepsilon/2 & \textrm{if $i\neq i'$.}
 \end{array} \right.
\end{displaymath}

In spirit, the structure $(X,\delta)$ is obtained by putting a copy of $\m{Y} \ (=(Z,\lceil d^{\m{Z}}\rceil_q))$ above a copy of $\m{Z}$ such that the distance between any point $(z,0) \in Z\times \{ 0\}$ and its counterpart $(z,1)$ in $Z\times \{ 1\}$ is $\varepsilon/2$. 

\begin{claimm}
The map $\delta$ is a metric on $X$. 
\end{claimm}

\begin{proof}

The maps $d^{\m{Z}}$ and $\lceil d^{\m{Z}}\rceil_q$ being metrics on $Z\times\{ 0\}$ and $Z\times \{1\}$, it suffices to verify that the triangle inequality is satisfied on triangles of the form $\{(x,0), (y,0), (z,1)\}$ and $\{(x,1), (y,1), (z,0)\}$, with $x, y, z \in Z$.

Assume that $x, y, z \in Z$, and consider the triangle $\{(x,1), (y,1), (z,0)\}$. Then \begin{align*}
\delta ((x,1) , (z,0)) & = d^{\m{Z}} (x,z) + \frac{\varepsilon}{2} \\
& \leq d^{\m{Z}} (x,y) + d^{\m{Z}} (y,z) + \frac{\varepsilon}{2} \\
& \leq \left\lceil d^{\m{Z}} (x,y) \right\rceil_q + d^{\m{Z}} (y,z) + \frac{\varepsilon}{2} \\
& \leq \delta ((x,1),(y,1)) + \delta ((y,1),(z,0)).
\end{align*}

Similarly, \[\delta ((y,1) , (z,0)) \leq \delta ((y,1),(x,1)) + \delta ((x,1),(z,0)).\] 

And finally, \begin{align*}
\delta ((x,1) , (y,1)) & = \left\lceil d^{\m{Z}} (x,y) \right\rceil_q \\
& \leq d^{\m{Z}} (x,y) + \frac{1}{q} \\
& \leq d^{\m{Z}} (x,y) + \frac{\varepsilon}{2} \\
& \leq d^{\m{Z}} (x,z) + d^{\m{Z}} (z,y) + \frac{\varepsilon}{2} \\
& \leq d^{\m{Z}} (x,z) + \frac{\varepsilon}{2} + d^{\m{Z}} (z,y) + \frac{\varepsilon}{2} \\
& \leq \delta ((x,1),(z,0)) + \delta ((z,0),(y,1)).
\end{align*}

Next, consider the triangle $\{(x,0), (y,0), (z,1)\}$. We have \begin{align*}
\delta ((x,0) , (z,1)) & = d^{\m{Z}} (x,z) + \frac{\varepsilon}{2} \\
& \leq d^{\m{Z}} (x,y) + d^{\m{Z}} (y,z) + \frac{\varepsilon}{2} \\
& \leq \delta ((x,0),(y,0)) + \delta ((y,0),(z,1)).
\end{align*}

Similarly, \[\delta ((y,0) , (z,1)) \leq \delta ((y,0),(x,0)) + \delta ((x,0),(z,1)).\] 

Finally, \begin{align*}
\delta ((x,0) , (y,0)) & = d^{\m{Z}} (x,y) \\
& \leq d^{\m{Z}} (x,z) + d^{\m{Z}} (z,y) \\
& \leq d^{\m{Z}} (x,z) + \frac{\varepsilon}{2} + d^{\m{Z}} (z,y) + \frac{\varepsilon}{2}\\
& \leq \delta ((x,0),(z,1)) + \delta ((z,1),(y,0)). \ \ \qedhere
\end{align*}

\end{proof}

Denote the space $(X,\delta)$ by $\m{X}$. Recall that every finite metric space embeds isometrically in $\Ur$. Hence, without loss of generality, we may suppose $Y \subset X \subset U$. We claim that $\m{Y}$ is as required. By construction, the space $\m{Y}$ is a finite subspace of $\Ur$ with distances in $\N /q$. Observe that $X \subset (Y)_{\varepsilon/2}$. Assume that a subspace $\mc{Y}$ of $\Ur$ is isometric to $\m{Y}$. By an argument similar to the one used in Claim \ref{claim:K}, the space $\m{X}$ embeds in $(\tilde{Y})_{\varepsilon/2}$. Thus, because $\m{Z}$ embeds in $\m{X}$, the set $(\tilde{Y})_{\varepsilon/2}$ contains a copy of $\m{Z}$, call it $\mc{Z}$. By Claim \ref{claim:K}, the set $(\tilde{Z})_{\varepsilon/2}$ contains a copy of $\m{K}$, call it $\mc{K}$. Then \[\tilde{K} \subset (\tilde{Z})_{\varepsilon/2} \subset ((\tilde{Y})_{\varepsilon/2})_{\varepsilon/2} \subset (\tilde{Y})_{\varepsilon}.\] This finishes the proof of Lemma \ref{lem:rich''}. 

\section{Proof of Theorem \ref{thm:UQ}} \label{section:UQ}

In this section, we think of $\Ur_{\Q}$ as a dense metric subspace of $\Ur$. We fix a partition of $\Ur_{\Q}$ as well as $\varepsilon > 0$, and we assume that there is a finite metric subspace $\m{Y}$ of $\Ur_{\Q}$ that does not embed in $B$. Our goal is to show that $\Ur_{\Q}$ embeds in $(R)_{\varepsilon}$. We start with the following technical lemma: 

\begin{lemma}

\label{thm: s_Q in s}

Let $\m{V}$ be a countable subspace of $\Ur$ with rational distances. Then for every $\varepsilon>0$ the  subspace $\Ur_{\Q} \cap (V)_{\varepsilon}$ includes a copy of $\m{V}$.
\end{lemma}

Assuming this result for the moment, here is how we prove Theorem \ref{thm:UQ}: let $U_{\Q} = B \cup R$ and $\varepsilon > 0$, and assume that there is a finite metric subspace $\m{Y}$ of $\Ur_{\Q}$ that does not embed in $B$. We wish to show that $\Ur_{\Q}$ embeds in $(R)_{\varepsilon}$. Choose $q \in \N$ large enough so that $2/q\leq \varepsilon$ and all distances appearing in $\m{Y}$ are integer multiples of $1/q$. Working in $\Ur$, set \[ B^* = \{ x \in U : (\{ x\})_{1/2q} \cap U_{\Q} \subset B\},\] \[R^*=U \smallsetminus B^*= U \cap (R)_{1/2q}.\]

Consider the space $U_{\N/q} ^*$ coming from Lemma \ref{lem:rich'}. The partition $U = B^* \cup R^*$ induces a partition of $U_{\N/q} ^*$. Observe that by weak indivisibility of $\Ur_{\N}$, the space $\Ur_{\N/q}$ is weakly indivisible as well. We also claim that the space $\m{Y}$ does not embed in $U_{\N/q} ^* \cap B^*$. Indeed, otherwise, we could find a copy $\mc{Y}$ of $\m{Y}$ in $U_{\N/q} ^* \cap B^*$. Lemma \ref{thm: s_Q in s} applied to $\m{V} = \m{Y}$ would then guarantee that $U _{\Q} \cap (\tilde{Y})_{1/2q}$ contains a copy of $\m{Y}$. But by construction, $U _{\Q} \cap (\tilde{Y})_{1/2q} \subset B$. So $\m{Y}$ would embed in $B$, a contradiction. Therefore, by weak indivisibility of $\Ur_{\N/q}$, there is a subspace $\m{T}$ of $\Ur_{\N/q} ^* $ isometric to $\Ur_{\N/q}$ such that $T\subset R^*$. By construction of $\Ur_{\N/q} ^*$, the set $(T)_{1/q}$ includes an isometric copy of $\Ur$, hence an isometric copy $\mc{U}$ of $\Ur_{\Q}$. By Lemma \ref{thm: s_Q in s} applied to $\m{V}=\mc{U}$, the set $U_{\Q} \cap (\tilde{U})_{1/2q}$ contains a copy of $\Ur_{\Q}$. Observe now that: \[ U_{\Q} \cap (\tilde{U})_{1/2q} \subset ((T)_{1/q})_{1/2q} \subset ((R^*)_{1/q})_{1/2q} \subset (((R)_{1/2q})_{1/q})_{1/2q} \subset (R)_{2/q} \subset (R)_{\varepsilon}.\]

Theorem \ref{thm:UQ} is proved. The rest of this section is therefore devoted to a proof of Lemma \ref{thm: s_Q in s}.

\begin{proof}[Proof of Lemma \ref{thm: s_Q in s}]

The proof of this lemma closely follows the proof of \cite{LANVT} Proposition 2, and is included here for completeness. We construct the required copy of $\m{V}$ inductively. Consider $\{ y_n : n \in \N \}$ an enumeration of $\m{V}$. For $k \in \N$, set \[\delta _k = \frac{\varepsilon}{2} \sum
_{i = 0} ^k \frac{1}{2^i}.\] 

Set also
 \[\eta _k = \frac{\varepsilon}{3} \frac{1}{2^{k+1}}.\]
 
$\Ur_{\Q}$ being dense in $\Ur$, choose $z_0 \in \Ur _{\Q}$ such that $d^{\Ur}(y_0 , z_0) < \delta _0$.
Assume now that $z_0,\dots, z_n \in \Ur _{\Q}$ were constructed such that for every $k, l \leqslant
n$
\begin{displaymath}
\left \{ \begin{array}{l}
 d^{\Ur}(z_k , z_l)=d^{\Ur} (y_k , y_l), \\
 d^{\Ur} (z_k , y_k) < \delta _k.
 \end{array} \right.
\end{displaymath}

Again by denseness of $\Ur_{\Q}$ in $\Ur$, fix $z \in \Ur _{\Q}$ such that

\begin{center}
$d^{\Ur} (z,y_{n+1}) < \eta _{n+1}$.
\end{center}

Then for every $k \leqslant n$,
\begin{align*}
\left| d^{\Ur} (z , z_k) - d^{\Ur} (y_{n+1} , y_k) \right| & =  \left| d^{\Ur} (z , z_k) - d^{\Ur}(z_k
, y_{n+1})  + d^{\Ur}(z_k , y_{n+1})
 - d^{\Ur} (y_{n+1} , y_k) \right| \\
 & \leqslant  d^{\Ur} (z , y_{n+1}) + d^{\Ur} (z_k , y_k) \\
 & <  \eta _{n+1} + \delta _k \\
 & <  \eta _{n+1} + \delta_n.
\end{align*}

It follows that there is $z_{n+1} \in \Ur _{\Q}$ such that
\begin{displaymath}
\left \{ \begin{array}{l}
 \forall k \leqslant n \ \ d^{\Ur} (z_{n+1} , z_{k}) = d^{\Ur} (y_{n+1} , y_k) \\
 d^{\Ur} (z_{n+1} , z) < \eta _{n+1} + \delta _n.
 \end{array} \right.
\end{displaymath}

Indeed, consider the map $f$ defined on $\{ z_k : k \leqslant n \} \cup \{ z \}$ by:
\begin{displaymath}
\left \{ \begin{array}{l}
 \forall k \leqslant n \ \ f(z_k) = d^{\Ur} (y_{n+1} , y_k), \\
 f(z) =  \max\{\left| d^{\Ur} (z , z_k) - d^{\Ur} (y_{n+1} , y_k) \right| : k \leqslant n \}.
 \end{array} \right.
\end{displaymath}

\begin{claimm}
$f$ is Kat\v{e}tov.
\end{claimm}

\begin{proof}

The metric space $\{y_k: k \leqslant n+1 \}$ witnesses that $f$ is Kat\v{e}tov over the set $\{ z_k : k \leqslant n \}$ so it suffices to prove that for every $k\leqslant n$, \[ \left| f(z) - f(z_k) \right| \leqslant d^{\Ur}(z,z_k) \leqslant f(z) + f(z_k).\] 

Equivalently, for every $k\leqslant n$, \[ \left| d^{\Ur}(z,z_k) - f(z_k) \right| \leqslant f(z) \leqslant d^{\Ur}(z,z_k) + f(z_k).\] 

There is nothing to do for the left-hand side because by definition of $f$, we have \[ f(z) =  \max\{\left| d^{\Ur} (z , z_k) - f(z_k) \right| : k \leqslant n \}.\] 

For right-hand side, what we need to show is that for every $k,l\leqslant n$, \[ \left|d^{\Ur} (z , z_l) - d^{\Ur} (y_{n+1} , y_l)\right| \leqslant d^{\Ur} (z , z_k) + d^{\Ur} (y_{n+1} , y_k).\]

Equivalently, 

\begin{displaymath}
\left \{ \begin{array}{l}
 d^{\Ur} (z , z_l) - d^{\Ur} (y_{n+1} , y_l) \leqslant d^{\Ur} (z , z_k) + d^{\Ur} (y_{n+1} , y_k), \\
 d^{\Ur} (y_{n+1} , y_l) - d^{\Ur} (z , z_l) \leqslant d^{\Ur} (z , z_k) + d^{\Ur} (y_{n+1} , y_k).
 \end{array} \right.
\end{displaymath}

The first inequality is equivalent to \[ d^{\Ur} (z , z_l) - d^{\Ur} (z , z_k) \leqslant d^{\Ur} (y_{n+1} , y_k) + d^{\Ur} (y_{n+1} , y_l).\]

But this is satisfied because \[ d^{\Ur} (z , z_l) - d^{\Ur} (z , z_k) \leqslant d^{\Ur}(z_l , z_k) = d^{\Ur}(y_k,y_l) \leqslant d^{\Ur} (y_k , y_{n+1}) + d^{\Ur} (y_{n+1} , y_l).\]

Similarly, the second inequality is equivalent to \[ d^{\Ur} (y_{n+1} , y_l) - d^{\Ur} (y_{n+1} , y_k) \leqslant d^{\Ur} (z , z_k) + d^{\Ur} (z , z_l).\]

This holds because \[ d^{\Ur} (y_{n+1} , y_l) - d^{\Ur} (y_{n+1} , y_k) \leqslant d^{\Ur}(y_k,y_l) = d^{\Ur}(z_k,z_l) \leqslant d^{\Ur} (z , z_k) + d^{\Ur} (z , z_l). \qedhere\]

\end{proof}

The map $f$ being Kat\v{e}tov, consider a point $z_{n+1} \in \Ur _{\Q}$ realizing $f$ over the set $\{ z_k : k \leqslant n \} \cup \{ z \}$. Observe then that
\begin{eqnarray*}
d^{\Ur}(z_{n+1} , y_{n+1}) & \leqslant & d^{\Ur}(z_{n+1},z) + d^{\Ur}(z,y_{n+1})\\
& < & \eta _{n+1} + \delta _n + \eta _{n+1}\\
& < & \delta _{n+1}.
\end{eqnarray*}

After infinitely many steps, we are left with $\{ z_n : n \in \N \} \subset \Ur _{\Q} \cap
(V)_{\varepsilon}$ isometric to $\m{V}$. \end{proof}

\section{Age-indivisibility may not imply weak indivisibility} \label{section:age not weak}

In what follows, the set $S$ is a fixed dense subset of $[0,2]$. Let $\E _S$ be the class of all finite metric spaces $\m{X}$ with distances in $S$ which embed isometrically into the unit sphere $\mathbb{S} ^{\infty}$ of $\ell _2$ with the property that $\{ 0_{\ell _2} \} \cup \m{X}$ is affinely independent. 

\begin{claimm}
There is a unique countable ultrahomogeneous metric space $\s _S$ whose class of finite metric subspaces is exactly $\E _S$. Moreover, the metric completion of $\s _S$ is $\s$.
\end{claimm}

\begin{proof}
See \cite{NVT} or \cite{NVT'}. 
\end{proof}

We show: 

\begin{thm}
\label{thm:ageind}
The space $\s _S$ is age-indivisible. 
\end{thm}

We also indicate why the space $\s _S$ may not be weakly indivisible. The proof of those results are provided in Subsection \ref{subsection:1} and Subsection \ref{subsection:2} respectively. 

\subsection{The space $\s _S$ is age-indivisible} \label{subsection:1}

Let $\m{Y}$ be a finite metric subspace of $\s_S$. We need to show: 

\begin{claimm}
\label{claim:ageind}
There is a finite metric subspace $\m{Z}$ of $\s_S$ such that for every partition $Z=B\cup R$, the space $\m{Y}$ embeds in $B$ or $R$.
\end{claimm}

The main ingredient of the proof is the following deep result due to Matou\v{s}ek and R\"odl: 

\begin{thm}[Matou\v{s}ek-R\"odl \cite{MR}]
\label{thm:MR}
Let $\m{X}$ be an affinely independent finite metric subspace of $\s$ with circumradius $r$, and let $\alpha > 0$. Then there is a finite metric subspace $\m{Z}$ of $\s$ with circumradius $r+\alpha$ such that for every partition $Z=B\cup R$, the space $\m{X}$ embeds in $B$ or $R$.
\end{thm}

What we need to prove is that in the case where $\m{X}=\m{Y}$, we may arrange $\m{Z}$ to be a subspace of $\s _S$ (that is, with distances in $S$ and $\{0_{\ell_2} \} \cup \m{Z}$ affinely independent). We will make use of the following facts along the way:

\begin{thm}[Schoenberg \cite{S}]

\label{thm:S}
Let $X=\{ x_k:1\leq k \leq n\}$ be a finite set and let $\delta: \funct{X^2}{\R}$ satisfying:  
\begin{enumerate}
	\item for every $x \in X$, $\delta(x,x)=0$, 
	\item for every $x, x'\in X$, $\delta(x,x)=0$ and $\delta(x',x)=\delta(x,x')$.  
\end{enumerate}

Then $(X, \delta)$ is isometric to a subset of $\ell_2$ iff 

\[ \max \left\{ \sum_{1\leq i<j \leq n} \delta(x_i,x_j)^2 \lambda_i \lambda_j : \sum_{k=1} ^n \lambda_k ^2 = 1 \ \ \textrm{and} \ \ \sum_{k=1} ^n \lambda_k = 0 \right\} \leq 0 \enspace .\] 

Moreover, $(X, \delta)$ is isometric to an affinely independent subset of $\ell_2$ iff the inequality is strict. 

\end{thm}

\begin{claimm}
\label{claim:Schoenberg}
Let $\m{X}$ be a finite affinely independent metric subspace of $\s$ with circumradius $r$. Then there is $\varepsilon >0$ such that for every $\delta: \funct{X^2}{\R}$ satisfying  
\begin{enumerate}
	\item for every $x, x'\in X$, $\delta(x,x)=0$ and $\delta(x',x)=\delta(x,x')$,  
	\item $|\delta^2 - (d^{\m{X}})^2| < \varepsilon^2$, 
\end{enumerate}
the space $(X, \delta)$ is an affinely independent metric subspace of $\s$. 
\end{claimm}

\begin{proof}
Direct from Theorem \ref{thm:S} and from the fact that the map $M\mapsto Q_M$ is continuous, where for a matrix $M=(m_{ij})_{1\leq i, j \leq n}$, \[ Q_M = \max \left\{ \sum_{1\leq i<j \leq n} m_{ij}\lambda_i \lambda_j : \sum_{k=1} ^n \lambda_k ^2 = 1 \ \ \textrm{and} \ \ \sum_{k=1} ^n \lambda_k = 0 \right\}. \qedhere \] 

\end{proof}

\begin{claimm}
\label{claim:cir}
Let $\m{X}$ be a finite metric subspace of $\s$ with circumradius $r$. Let $\varepsilon > 0$. Then $(X, \sqrt{(d^{\m{X}})^2+\varepsilon^2})$ is Euclidean, affinely independent with circumradius at most $r+\varepsilon$. 
\end{claimm}

\begin{proof}
Let $V$ be the affine space spanned by $X$. Choose $(e_x)_{x\in X}$ a family of pairwise orthogonal unit vectors in $V^{\perp}$. For $x\in X$, set $\tilde{x}=x+\varepsilon/\sqrt{2} \ e_x$. Then the set $\{\tilde{x} : x \in X\}$ is affinely independent and is isometric to $(X, \sqrt{(d^{\m{X}})^2+\varepsilon^2})$. Its circumradius is at most $r+\varepsilon$ because it is contained in the ball centered at the circumcenter of $X$ and with radius $r+\varepsilon$.
\end{proof}

\begin{claimm}
\label{claim:affind}
Let $\m{X}$ be an affinely independent subspace of $\s$. Then $\m{X}\cup\{ 0_{\ell_2}\}$ is affinely independent iff the circumradius of $\m{X}$ is $<1$.
\end{claimm}
  
\begin{proof}

Let $V$ be the affine space spanned by $X$. Then the set $\s\cap V$ is the circumscribed sphere of $X$ in $V$. It has radius $<1$ iff $0_{\ell_2}$ does not belong to $V$. \qedhere \end{proof}

\begin{proof}[Proof of Claim \ref{claim:ageind}]

First, we show that there is an affinely independent finite metric subspace $\m{Z}_0$ of $\s$ with circumradius $<1$ such that for every partition $Z_0=B\cup R$, $\m{Y}$ embeds in $B$ or $R$: 

Let $r$ denote the circumradius of $\m{Y}$. Because $\m{Y}$ is a subspace of $\s _S$, the space $\m{Y}\cup\{ 0_{\ell_2}\}$ is affinely independent and by Claim \ref{claim:affind}, we have $r<1$. By Claim \ref{claim:Schoenberg}, fix $\varepsilon > 0$ such that $r+2\varepsilon < 1$ and such that for every map $\delta: \funct{Y^2}{\R}$ satisfying 
\begin{enumerate}
	\item for every $y, y'\in Y$, $\delta(y,y)=0$ and $\delta(y',y)=\delta(y,y')$,
	\item $|\delta^2 - (d^{\m{Y}})^2| < \varepsilon^2$,
\end{enumerate}
the space $(Y, \delta)$ is still Euclidean and affinely independent. Fix $\alpha>0$ such that $\alpha~<~\varepsilon$. By choice of $\alpha$, the space $(Y, \sqrt{(d^{\m{Y}})^2-\alpha^2})$ is still Euclidean and affinely independent. It should be clear that its circumradius is at most $r$. Apply Theorem \ref{thm:MR} to produce a finite metric subspace $\m{T}$ of $\s$ with circumradius $r+\alpha$ such that for every partition $T=B\cup R$, the space $(Y, \sqrt{(d^{\m{Y}})^2-\alpha^2})$ embeds in $B$ or $R$. Set $\m{Z}_0=(T, \sqrt{(d^{\m{T}})^2+\alpha^2})$. We claim that $\m{Z}_0$ is as required.

Indeed, by Claim \ref{claim:cir}, $\m{Z}_0$ is Euclidean, affinely independent, and its circumradius is at most $ (r + \alpha) + \alpha < r + 2 \varepsilon < 1$. Next, if $Z_0=B\cup R$, this partition induces a partition $T=B\cup R$. By construction of $\m{T}$, there is a subspace $\mc{Y}$ of $\m{T}$ isometric to $(Y, \sqrt{(d^{\m{Y}})^2-\alpha^2})$ contained in $B$ or $R$. Note that in $\m{Z}_0$, the metric subspace supported by $\tilde{Y}$ is isometric to \[ \left(Y, \sqrt{\left(\sqrt{(d^{\m{Y}})^2-\alpha^2}\right)^2+\alpha^2 }\right) = \left(Y, \sqrt{(d^{\m{Y}})^2-\alpha^2+\alpha^2 }\right) = (Y, d^{\m{Y}}) = \m{Y}.\] 


Consider the space $\m{Z}_0$ we just constructed. Using Claim \ref{claim:Schoenberg} as well as the denseness of $S$, we may modify slightly all the distances in $\m{Z}_0$ that are not in $S$ and turn $\m{Z}_0$ into an affinely independent subspace $\m{Z}$ of $\s$ with distances in $S$ and circumradius $<1$. By Claim \ref{claim:affind}, the space $\{0_{\ell_2} \} \cup \m{Z}$ is affinely independent. Therefore, $\m{Z}$ embeds in $\s_S$. Finally, note that since all the distances of $\m{Z}_0$ that were already in $S$ did not get changed, the copies of $\m{Y}$ in $\m{Z}_0$ remain unaltered when passing to $\m{Z}$. It follows that for every partition $Z=B\cup R$, the space $\m{Y}$ embeds in $B$ or $R$. \qedhere \end{proof}

\subsection{The space $\s _S$ may not be weakly indivisible} \label{subsection:2}

The starting point of this section is the following theorem: 

\begin{thm}[Odell-Schlumprecht \cite{OS}]
\label{thm:OS'}
There is a partition $\s=B\cup R$ and $\varepsilon>0$ such that 
\begin{enumerate}
	\item For every linear subspace $V$ of $\ell_2$ with $\dim V = \infty$, $\s \cap V \not\subset (B)_{\varepsilon}$. 
	\item For every linear subspace $V$ of $\ell_2$ with $\dim V = \infty$, $\s \cap V \not\subset (R)_{\varepsilon}$. 
\end{enumerate}
\end{thm}

In response to an inquiry of the authors, Thomas Schlumprecht \cite{OS'} indicated that the method that was used to prove Theorem \ref{thm:OS'} in \cite{OS} (where the statement is proved first in another Banach space known as the Schlumprecht space, and then transferred to $\ell_2$), can be adapted to show that $\dim V = \infty$ may be replaced by $\dim V =2$ in (i). However, he indicated recently that some obstruction had appeared. Nevertheless, we would like to present here how the aforementioned strengthening of Theorem \ref{thm:OS'} implies that $\s_S$ is not weakly indivisible. 

\begin{thm}
\label{thm:OS}
Assume that there is a partition $\s=B\cup R$ and $\varepsilon>0$ such that 
\begin{enumerate}
	\item for every linear subspace $V$ of $\ell_2$ with $\dim V = 2$, $\s \cap V \not\subset (B)_{\varepsilon}$, 
	\item for every linear subspace $V$ of $\ell_2$ with $\dim V = \infty$, $\s \cap V \not\subset (R)_{\varepsilon}$. 
\end{enumerate}

Then $\s_S$ is not weakly indivisible. 
\end{thm}

Consider the partition of $\s$ provided by Theorem \ref{thm:OS}. It should be clear that it induces a partition of $\s_S$. 

\begin{claimm}
\label{claim:notweak}
$\s_S = B \cup R$ witnesses that $\s_S$ is not weakly indivisible. 
\end{claimm}

The proof makes use of the following fact, which we prove for completeness: 

\begin{claimm}
Let $Y \subset \s$ be isometric to $\s$. Then there is a closed linear subspace $V$ of $\ell_2$ with $\dim V = \infty$ such that $Y = V \cap \s$. 
\end{claimm}

\begin{proof}
Consider $V$ the closed linear span of $Y$ in $\ell_2$. Consider also the set $W = \{\lambda y : \lambda \in \R, \ y \in Y\}$. We will be done if we show $V=W$. Clearly, $W\subset V$. For the reverse inclusion, observe that because $Y$ is closed (it is isometric to a complete metric space), the set $W$ is closed. Therefore, it is enough to show that all the finite linear combinations of elements of $V$ that have norm $1$ are in $Y$, i.e. for every $y_1,\ldots,y_n \in Y$ and $\lambda_1, \ldots, \lambda_n \in \R$ such that $\sum_{i=1}^n \lambda_i y_i \neq 0_{\ell_2}$, \[ \frac{\sum_{i=1}^n \lambda_i y_i}{\left\|\sum_{i=1}^n \lambda_i y_i\right\|} \in Y.\]

We proceed by induction on $n$. For $n=2$, we first consider the case $\lambda_1=\lambda_2=1$. Note that $y_1$ and $y_2$ cannot be antipodal (otherwise $y_1+y_2=0_{\ell_2}$), and that $\frac{y_1+y_2}{\left\|y_1+y_2\right\|}$ can be characterized metrically in terms of $y_1$ and $y_2$. For example, it it is the unique geodesic middle point of $y_1$ and $y_2$ in the intrinsic metric on $\s$. Since the intrinsic metric can be defined in terms of the usual Hilbertian metric on $\s$, this point must belong to $Y$. By a usual middle-point-type argument, it follows that the entire geodesic segment between $y_1$ and $y_2$ is contained in $Y$. Using then that $Y$ is closed under antipodality (because $Y$ being isometric to $\s$ any $y\in Y$ must have a point at distance $2$), as well as a middle-point-type argument again, the entire great circle through $y_1$ and $y_2$ is contained in $Y$. That finishes the case $n=2$. Assume that the property is proved up to $n\geq 2$. Fix $y_1,\ldots,y_{n+1} \in Y$ and $\lambda_1, \ldots, \lambda_{n+1} \in \R$. Then writing \[z = \frac{\sum_{i=1}^n \lambda_i y_i}{\left\|\sum_{i=1}^n \lambda_i y_i \right\|} \enspace, \] the vector 
\[ \frac{\sum_{i=1}^{n+1} \lambda_i y_i}{\left\|\sum_{i=1}^{n+1} \lambda_i y_i\right\|} \] is a linear combination of $z$ and $y_{n+1}$ with norm 1. Therefore, it is of the form \[ \frac{\alpha z+\beta y_{n+1}}{\left\|\alpha z+\beta y_{n+1}\right\|}.\]

By induction hypothesis, $z$ is in $Y$. So again by induction hypothesis (case $n=2$), \[ \frac{\alpha z+\beta y_{n+1}}{\left\|\alpha z+\beta y_{n+1}\right\|} \in Y.\]

Therefore, \[ \frac{\sum_{i=1}^{n+1} \lambda_i y_i}{\left\|\sum_{i=1}^{n+1} \lambda_i y_i\right\|} \in Y. \qedhere\]

\end{proof}

\begin{proof}[Proof of Claim \ref{claim:notweak}]
Let $W$ be a linear subspace of $\ell_2$ with $\dim W=2$. By compactness of $\s\cap W$ and denseness of $\s_S$ in $\s$, there is $X\subset \s_S$ finite such that $\s\cap W \subset (X)_{\varepsilon}$. Let $\m{X}$ denote the metric subspace of $\s_S$ supported by the set $X$. Then $\m{X}$ does not embed in $B$ because otherwise, there would be a linear subspace $V$ of $\ell_2$ with $\dim V = 2$ such that $\s \cap V \subset (B)_{\varepsilon}$, violating (i) of Theorem \ref{thm:OS}. On the other hand, $\s_S$ cannot embed in $R$: let $Y\subset \s_S$ be isometric to $\s_S$. Then in $\s$, the closure $\bar{Y}$ of $Y$ is isometric to $\s$. By Claim \ref{claim:notweak}, there is a closed linear subspace $V$ of $\ell_2$ with $\dim V = \infty$ such that $\bar{Y} = V \cap \s$. By (ii) of Theorem \ref{thm:OS}, $\bar{Y} \not\subset (R)_{\varepsilon}$. Therefore $\bar{Y}\not\subset R$.  \qedhere 
\end{proof}


\begin{thebibliography}{DLPS07}

\bibitem[DLPS07]{DLPS}
C. Delhomm\'{e}, C. Laflamme, M. Pouzet and N. Sauer, Divisibility of countable metric spaces, \emph{Europ. J. Combinatorics}, 28 (6), 1746--1769, 2007. 

\bibitem[F00]{Fr}
R. Fra\"iss\'e, Theory of relations, Studies in Logic and the
Foundations of Mathematics, 145, North-Holland Publishing Co.,
Amsterdam, 2000. 

\bibitem[KPT05]{KPT}
A. S. Kechris, V. Pestov and S. Todorcevic, Fra\"iss\'e limits,
Ramsey theory, and topological dynamics of automorphism groups,
\emph{Geom. Funct. Anal.}, 15, 106--189, 2005.

\bibitem[LNPS-]{LNS}
C. Laflamme, L. Nguyen Van Th\'e, M. Pouzet and N. W. Sauer, Partitions and indivisibility properties of countable
infinite dimensional vector spaces, \emph{submitted}. 

\bibitem[LN08]{LANVT}
J. Lopez-Abad and L. Nguyen Van Th\'e, The oscillation stability problem for the Urysohn sphere: A combinatorial approach, \emph{Topology Appl.},  155 (14), 1516--1530, 2008.

\bibitem[MR95]{MR}
J. Matou\v sek and V. R\"odl, On Ramsey sets in spheres, \emph{J. Combin. Theory Ser. A}, 70, 30--44, 1995.


\bibitem[N07]{N1}
J. Ne\v{s}et\v{r}il, Metric spaces are Ramsey, \emph{European J. Combin.}, 28 (1), 457--468, 2007.

\bibitem[NVT06]{NVT}
L. Nguyen Van Th\'e, Th\'eorie de Ramsey structurale des espaces m\'etriques et dynamique topologique des groupes d'isom\'etries, Ph.D. Thesis, Universit\'e Paris 7, 2006, available in English at http://tel.archives-ouvertes.fr/tel-00139239/fr/. 

\bibitem[NVT-]{NVT'}
L. Nguyen Van Th\'e, Structural Ramsey theory of metric spaces and topological dynamics of isometry groups, \emph{Mem. Amer. Math. Soc.}, \emph{to appear}, preprint available at arXiv:0804.1593. 

\bibitem[NS09]{NVTS}
L. Nguyen Van Th\'e and N. W. Sauer, The Urysohn sphere is oscillation stable, \emph{GAFA, Geom. Funct. Anal.}, 19, 536ñ-557, 2009. 

\bibitem[OS94]{OS}
E. Odell and T. Schlumprecht, The distortion problem, \emph{Acta Mathematica}, 173, 259--281, 1994.

\bibitem[S08]{OS'}
T. Schlumprecht, personal communication. 

\bibitem[S38]{S}
I. J. Schoenberg, Metric spaces and positive definite functions, \emph{Trans. Amer. Math. Soc.}, 44 (3), 1938, 522--536.

\bibitem[U08]{Pr} A. Leiderman, V. Pestov, M. Rubin, S. Solecki, V.V. Uspenskij (eds.), 
\textit{Special issue: Workshop on the Urysohn space,
Held at Ben-Gurion University of the Negev, Beer Sheva, May 21--24, 2006},  
Topology Appl., 155 (14), 2008.


\end{thebibliography}
\end{document}